%% file: Skew_Howe_duality_for_crystals_and_the_cactus_group.tex
\documentclass[11pt]{amsart}

\usepackage{etex}
\usepackage{amscd}
\usepackage{amsmath}
\usepackage{amssymb}
\usepackage[T1,T2A]{fontenc}
\usepackage[utf8]{inputenc}
\usepackage{blkarray}
\usepackage{caption}
\usepackage{color}
\usepackage{enumerate}
\usepackage{float}
\usepackage{graphicx}
\usepackage{hyperref}
\usepackage{kbordermatrix}
\usepackage{epsfig}
\usepackage{multicol}
\usepackage{multirow}
\usepackage{mwe}
\usepackage{picins}
\usepackage{setspace}
\usepackage{tikz}
\usepackage{tikz-cd}
\usetikzlibrary{arrows, matrix, positioning,calc}
\usepackage[all]{xy}
\usepackage{ytableau}

\newtheorem{theorem}{Theorem}[section]
\newtheorem{lemma}[theorem]{Lemma}
\newtheorem{proposition}[theorem]{Proposition}
\newtheorem{corollary}[theorem]{Corollary}
\newtheorem{definition}[theorem]{Definition}
\newtheorem{example}[theorem]{Example}
\newtheorem{remark}[theorem]{Remark}

\newcommand{\nc}{\newcommand}	
\nc{\al}{\alpha}
\nc{\be}{\beta}
\nc{\ga}{\gamma}
\nc{\de}{\delta}
\nc{\ep}{\epsilon}
\nc{\la}{\lambda}
\nc{\si}{\sigma}
\nc{\fg}{\mathfrak g}
\nc{\fh}{\mathfrak h}
\nc{\msl}{\mathfrak{sl}}
\nc{\mgl}{\mathfrak{gl}}
\nc{\mb}{\mathcal{B}}
\renewcommand{\C}{\mathbb{C}}
\nc{\R}{\mathbb{R}}
\nc{\so}{s^{\text{o}}_{1,k}}
\nc{\sii}{s^{\text{i}}_{1,k}}
\nc{\wt}{\text{wt}}
\nc{\om}{\omega}
\nc{\Z}{\mathbb{Z}}

\newcommand {\DS} [1] {\displaystyle #1}

\hypersetup{pdfstartview={XYZ null null 1.00}}

\title{Skew Howe duality for crystals and the cactus group}

\author[Iva Halacheva]{Iva Halacheva}
\address{i.halacheva@northeastern.edu \\ Northeastern University, Department of Mathematics;\\
	360 Huntington Ave, Boston, MA, USA 02115}
\begin{document}

	\begin{abstract} 
	The  crystals for a finite-dimensional complex reductive Lie algebra $\fg$ encode the structure of its representations, yet can also reveal surprising new structure of their own. We study the cactus group $C_{\fg}$, constructed using the Dynkin diagram of $\fg$, and its combinatorial action on any $\fg$--crystal via Sch\"{u}tzenberger involutions. We compare this action with that of the Berenstein--Kirillov group on Gelfand--Tsetlin patterns. Henriques and Kamnitzer define an action of $C_n=C_{\mgl_n}$ on $n$--tensor products of $\fg$--crystals, for any $\fg$ as above. We discuss the crystal corresponding to the $\mgl_n \times \mgl_m$--representation $\Lambda^N(\C^n \otimes \C^m)$, derive skew Howe duality on the crystal level and show that the two types of cactus group actions agree in this setting. A future application of this result is discussed in studying two families of maximal commutative subalgebras of the universal enveloping algebra, the shift of argument and Gaudin algebras, where an algebraically constructed monodromy action matches that of the cactus group.
	\end{abstract}
	
	\maketitle


	\section{Introduction}
	\label{sec:intro}
	One approach to studying the representation theory of a finite--dimensional complex reductive Lie algebra $\fg$ is through crystals--a combinatorial tool which encodes the structure of $\fg$--representations.  A crystal can be presented as a coloured directed graph exhibits nice behaviour for irreducible components, direct sums, and tensor products. The category of $\fg$--crystals was studied by Henriques and Kamnitzer in \cite{HK}, where they describe its structure as a \textsl{coboundary category} with an ``outer'' action of the cactus group $C_n$ on $n$--tensor products of $\fg$--crystals, analogous to the action of the braid group $B_n$ in a braided category. We study the cactus group $C_{\fg}$ for any $\fg$, with $C_{\fg}=C_n$ for $\fg=\mgl_n(=\mgl_n(\C))$, and an ``inner'' action on any $\fg$--crystal via generalized Sch\"{u}tzenberger involutions. In \cite{L}, Losev constructs an action of $C_{\fg}$ on the Weyl group $W_\fg$, compatible with Kazhdan-Lusztig cells, using category $\mathcal{O}$ wall-crossing functors. For $\fg=\mgl_n$ (where for the zero--weight space of the tensor product of crystals with highest weight $\varpi_1$, $B(\varpi_1)^{\otimes n}_0 \cong S_n$), this action coincides with the Henriques--Kamnitzer $C_n$--action on $n$--fold tensor products of $\fg$-crystals. In \cite{BGL}, the authors define and study a cactus action on integrable highest weight $U_q(\fg)$--modules, and in particular pose the question of how the ``inner'' and ``outer'' cactus actions are related, both on the module and crystal level.  We answer the question on the level of crystals in setting of the $\mgl_n \times \mgl_m$--representation $\Lambda^N(\C^n \otimes \C^m)$. The corresponding crystal can be realized as the set $\Lambda^NB_{n,m}$ of $n\times m$--matrices with $N$ 1's and all remaining entries 0's. We show this set has the structure of both a $\mgl_n$-- and a $\mgl_m$--crystal, and the two structures commute.  We then derive $(\mgl_n,\mgl_m)$ skew Howe duality on the level of crystals, also studied by van Leeuwen \cite{vL} and Danilov-Koshevoy \cite{DK}, analogous to skew Howe duality on the level of representations. We use this duality and the decomposition:
	\[\Lambda^N B_{n,m} \cong \bigsqcup_{\underline{l} \in \mathbb{N}^n, |\underline{l}|=N} \mathcal{B}^{\mgl_m}_{\varpi_{l_1}} \otimes \hdots \otimes \mathcal{B}^{\mgl_m}_{\varpi_{l_n}} \] 
	where $\varpi_i$ is the $i$--th fundamental weight of $\mgl_m$, to show that the ``outer'' action of the cactus group $C_n$ on $\Lambda^NB_{n,m}$, acting on $n$--tensor products of $\mgl_m$--crystals as defined by Henriques and Kamnitzer, agrees with the ``inner'' cactus group action on this $\mgl_n$--crystal. A future application of this result is to relate the monodromy actions (studied in \cite{HKRW}) coming from two  families of maximal commutative subalgebras in $U(\mgl_n)$ and $(U(\mgl_m))^{\otimes n}$, known as shift of argument and Gaudin algebras.
	
	Section \ref{sec:prelim} provides the setup and some background information on crystals, while Section \ref{sec:cactus} describes the cactus group $C_{\fg}$ for general $\fg$ and its action on any $\fg$--crystal. In Section \ref{sec:tensor}, we discuss the crystal structure of $\Lambda^NB_{n,m}$. The result of establishing crystal skew Howe duality and the equivalence of the outer and inner cactus group actions on $\Lambda^NB_{n,m}$ is in Section \ref{sec:skew}. Finally, a future application in a geometric realization of the cactus group action is described in Section \ref{sec:shiftgaudin}.

	\subsection{Acknowledgements}
		I would like to thank Joel Kamnitzer for suggesting studying the cactus group and for many discussions. I am also grateful to Alex Weekes and the Geometric Representation Theory group at the University of Toronto for the mathematical discussions. This project was partially supported by an NSERC CGS-D scholarship.

\section{Preliminaries}
\label{sec:prelim}
Crystals are a useful combinatorial tool for encoding the  representations of a complex reductive (in particular, finite--dimensional) Lie algebra $\fg$ and were originally introduced by Kashiwara in \cite{Kash0} and \cite{Kash}. Let $U_q(\fg)$ denote the quantum group of $\fg$, and $I$ the Dynkin diagram of $\fg$. We present some of the necessary background here, further information can be obtained in \cite{HoKa}, \cite{{Jo}}, and \cite{Kash}. Intuitively, a crystal $B=B_V$ is a ``nice'' basis for a $U_q(\fg)$--module $V$ ``at $q=0$'', together with operators $e_i,\;f_i : B \longrightarrow B\cup \{0\}, \; i \in  I$, called the \textbf{Kashiwara operators}, which are modified versions of the root operators.

Kashiwara's \textsl{grand loop argument} \cite{Kash} shows that for every finite-dimensional $U_q(\fg)$--representation, which in turn corresponds to a finite-dimensional $\fg$--representation, there is a crystal which is unique up to isomorphism.

\vspace{0.3cm}
\noindent\textbf{\underline{Properties}:} Crystals have several nice properties which make them well-suited for studying the structure of representations of $\fg$.
\begin{enumerate}
	\item Let $\Lambda$ denote the weight lattice of $\fg$. For a $\fg$--representation $V=\bigoplus_{\la \in \Lambda}V^{\la}$ with corresponding crystal $B=\sqcup_{\la \in \Lambda} B^{\la}$, the character of $V$ can be recovered by:\[\text{ch}(V)=\sum_{\la \in \Lambda} |B^{\la}| e^{\la}\]	
	\parpic[r][b]{%
		\begin{minipage}{45mm}
			\begin{figure}[H]
				\begin{center}
					\def\svgwidth{3.7cm}
					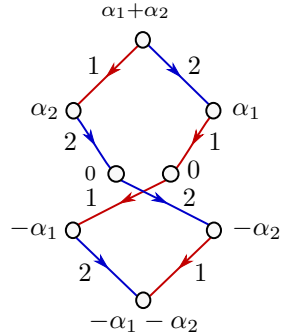
				\end{center}
				\vspace{0.2cm}
				\captionsetup{width=0.9\textwidth}
				\caption{The crystal graph for the adjoint representation of $\msl_3$.}
				\label{fig:crystal}
			\end{figure}
		\end{minipage}}
		
		\item A useful combinatorial tool that comes from a crystal $B$ is the \textbf{crystal graph}, whose vertices are the set $B$ and whose edges are $I$--coloured and correspond to the action of the Kashiwara operators: for $b_1,b_2 \in B$, we have $b_1 \xrightarrow{i} b_2$ precisely when $f_i(b_1)=b_2$ (equivalently $e_i(b_2)=b_1$). For instance, the crystal graph for the adjoint representation of $\msl_3$ is shown in Figure \ref{fig:crystal}. Moreover, irreducible components of the representation are easy to spot in the crystal graph as they correspond to connected components of the graph.
		
		More abstractly, a \textbf{$\fg$--crystal} $B$ is a finite set $B$ together with the maps:
		\begin{align*}
			&e_i,\; f_i: B \rightarrow B\cup\{0\}, \; \ep_i, \phi_i: B \rightarrow \Z \quad (i \in I) \\
			&\hspace{2cm}\wt: B \rightarrow \Lambda
		\end{align*}
		Satisfying the properties:
		\begin{enumerate}
			\item For all $b,c \in B$, $f_i(b)=c$ if and only if $b=e_i(c)$.
			\item For $b \in B$, if $e_i(b) \in B$, then $\wt(e_i(b))=\wt(b)+\al_i$ and if $f_i(b) \in B$, then $\wt(f_i(b))=\wt(b)-\al_i$.
			\item For all $b \in B, \; i \in I$, $\ep_i(b)=\max\{k\;| e^k_i(b) \in B\}$ and $\phi_i(b)=\max\{k\;| f^k_i(b) \in B\}$.
			\item For all $b \in B, \; i \in I$, we have $\phi_i(b)-\ep_i(b)=\left\langle \wt(b),\alpha^{\vee}_i\right\rangle$.
		\end{enumerate}
	\end{enumerate}
	\vspace{0.2cm}
	\begin{enumerate}\setcounter{enumi}{2}
		\item Taking the \textbf{tensor product} of two $\fg$--crystals $A$ and $B$ can be done using the following rules--the underlying set of $A \otimes B$ is $A \times B$ and for any element $a \otimes b \in A \otimes B$ we have:

			\begin{align*}
				e_i(a \otimes b) &=\begin{cases}
					e_i(a) \otimes b & \text{ if } \phi_i(a) \geq \ep_i(b) \\
					a \otimes e_i(b) & \text{ otherwise}
				\end{cases} \\
				f_i(a \otimes b) &=\begin{cases}
					f_i(a) \otimes b & \text{ if } \phi_i(a) > \ep_i(b) \\
					a \otimes f_i(b) & \text{ otherwise}
				\end{cases} \\
				\wt(a\otimes b) &=\wt(a) + \wt(b) \\
				\ep_i(a\otimes b) &= \max(\ep_i(a),\ep_i(b)-\left\langle\wt(a),\al^{\vee}_i\right\rangle) \\
				\phi_i(a\otimes b) &= \max(\phi_i(b),\phi_i(a)+\left\langle\wt(b),\al^{\vee}_i\right\rangle) \\
			\end{align*}

		We will discuss more generally the tensor product of $m \geq 2$ crystals in Section \ref{sec:tensor}. In addition, the \textbf{direct sum} of two crystals is their disjoint union.
		
		\vspace{0.2cm}
		
		\item 	Among the abstract $\fg$--crystals, we wish to only consider those corresponding to a representation of $\fg$. For that purpose, we take the unique \textsl{closed family} of \textbf{highest weight crystals} $\{B_{\la}| \; \la \in \Lambda_+\}$. Namely, in each $B_{\la}$ there is a (highest weight) element $b_{\la}$ such that $\wt(b_{\la})=\la, \; e_i(b_{\la})=0 \; \forall \; i \in I$, and $B_{\la}$ is generated by $f_i$ acting on $b_{\la}$, together with the additional property that the connected crystal generated by $b_{\la} \otimes b_{\mu} \in B_{\la} \otimes B_{\mu}$ is isomorphic to $B_{\la+\mu}$, for all $\la, \mu \in \Lambda_+$, i.e. dominant weights (see \cite{Jo}). Then, the \textbf{category of $\fg$--crystals} we want to consider has objects each of whose connected components is $B_{\la}$ for some $\la \in \Lambda_+$, while the morphisms are defined below.
	\end{enumerate}
	
	\begin{definition}
		Suppose $\mb_1$ and $\mb_2$ are $\fg$--crystals. A \textbf{crystal morphism} $F: \mb_1 \longrightarrow \mathcal{B}_2$ is a map of sets $F: \mb_1 \cup \{0\} \longrightarrow \mb_2 \cup \{0\}$ such that:
		\begin{enumerate}
			\item $F(0)=0$.
			\item If $b \in \mb_1$ s.t. $F(b) \in \mb_2$, then $\wt(F(b))=\wt(b), \; \ep_i(F(b))=\ep_i(b),\; \phi_i(F(b))=\phi_i(b)$.
			\item If $b \in \mb_1$ such that $F(b) \in \mb_2$, and $f_i(b) \in \mb_1$ (i.e. it is nonzero), then
			$F(f_i(b))=f_i(F(b))$, respectively if $e_i(b) \in \mb_1$ then $F(e_i(b))=e_i(F(b))$.	
		\end{enumerate}
	\end{definition}
	
	We will recall some further properties of crystals in the context of the following sections. Next, we define the general cactus group and its action on crystals.

\section{The Cactus Group}
\label{sec:cactus}

As before, $\fg$ denotes a (finite--dimensional) complex reductive Lie algebra, $I$ is the set of vertices in its Dynkin diagram, and $\{\alpha_i\}_{i \in I}$ are the simple roots. As described in \cite{HK}, there is a Dynkin diagram automorphism $\theta: I\longrightarrow I$ defined by $\alpha_{\theta(i)} = -w_0 \cdot \alpha_i$, where $w_0$ is the longest element in the Weyl group of $\fg$. For instance, in the case of $\mgl_n$, $\theta(i)=n-i$.

The \textbf{cactus group} $C_{\fg}$ corresponding to $\fg$ has generators $s_J$, where $J$ runs over the connected subdiagrams of $I$, and relations:
\vspace{0.2cm}
\begin{enumerate}
	\item\label{rel:invol} $s^2_J=1 \quad \forall \, J \subset I$
	\item\label{rel:parcomm} $s_Js_{J'} = s_{\theta_J(J')}s_J \quad  \forall \, J' \subset J$
	\item $s_Js_{J'}=s_{J'}s_J \quad \forall \, J',J \subset I$ such that $J' \cup J$ is not connected
\end{enumerate}
\vspace{0.2cm}

In Relation (\ref{rel:parcomm}), $\theta_J$ is the Dynkin diagram automorphism of $J$, namely  $\alpha_{\theta_J(j)} = -w^J_0 \cdot \alpha_j, \; \forall \; j \in J$, where $w^J_0$ is the longest element in the Weyl group for $\fg_J$ ($\fg$ restricted to $J$). In our original notation, $\theta=\theta_I$. This definition is a generalization of the cactus group defined in \cite{HK}. Henriques and Kamnitzer's cactus group corresponds to $C_{\mgl_n}$.

\begin{remark}
	The cactus group $C_{\fg}$ is a (distant) cousin of the braid group as it also surjects onto the Weyl group $W_{\fg}$ of $\fg$, by mapping $s_J \mapsto w^J_0$. The kernel of this surjection is called the \textbf{pure cactus group} $PC_{\fg}$:
	\[1 \longrightarrow PC_{\fg} \longrightarrow C_{\fg} \longrightarrow W_{\fg} \longrightarrow 1\]
\end{remark}

\subsection{An action on crystals}
\label{sec:action}

\subsubsection{Action by Sch\"{u}tzenberger involutions}

Let $B_{\la}$ be the highest weight $\fg$--crystal (corresponding to the irreducible highest weight representation $V_{\la}$) with highest weight $\lambda$ and crystal maps $\text{wt},e_i,f_i,\ep_i,\phi_i$, where $i \in I$. Then, as in \cite{HK}, we denote by $\xi_{\la}: B_{\la} \rightarrow B_{\la}$ the unique map of sets satisfying, for all $b \in B_{\la}$:
\begin{align*}
	e_i \cdot \xi_{\la}(b) &= \xi_{\la}(f_{\theta(i)} \cdot b) \\
	f_i \cdot \xi_{\la}(b) &= \xi_{\la}(e_{\theta(i)} \cdot b) \\
	\text{wt}(\xi_{\la}(b)) &= w_0 \cdot \text{wt}(b)
\end{align*}

In particular, $\xi_{\la}$ interchanges the highest and lowest weight elements of $B_{\la}$, $b_{\la}\stackrel{\xi_{\la}}{\longrightarrow} b^{\text{low}}_{\la}\stackrel{\xi_{\la}}{\longrightarrow} b_{\la}$. More generally we have:
\begin{definition}\label{def:schu}
	The \textbf{Sch\"{u}tzenberger involution} $\xi=\xi_B: B \rightarrow B$ for any $\fg$-crystal $B$ acts by applying the corresponding map $\xi_{\lambda}$ to each connected component $B_{\la}$ of $B$. 
\end{definition}

In \cite{HKRW}, together with Kamnitzer, Rybnikov, and Weekes, we show that the cactus group $C_{\fg}$ acts on any $\fg$--crystal.

\begin{theorem}[\cite{HKRW}, Theorem 5.19]\label{prop:action}
	There is a natural action of the cactus group $C_{\fg}$ on any $\fg$--crystal $B$ given by, for any connected subdiagram $J \subset I$:
	\[s_J(b)=\xi_{B_J}(b)\]
	
	\noindent where $B_J$ is the restriction of $B$ to the subdiagram $J$ of $I$; as a graph it has the same set of vertices but only the arrows labelled by $J$.
\end{theorem}

\subsubsection{The reduced cactus group}

In \cite{Kash} and \cite{Kash2}, Kashiwara studies the Sch\"{u}tzenberger involutions associated to a single node in the Dynkin diagram $I$ of $\fg$ and shows that they generate an action of the Weyl group on any $\fg$--crystal $B$. Namely, for $i \in I$, he defines a set automorphism $s_{\{i\}}: B \rightarrow B$ by, for any $b \in B$:
\[s_{\{i\}}(b) =\begin{cases}
f^{\left\langle \wt(b),\alpha^{\vee}_i\right\rangle}_i(b) & \text{ if } \left\langle \wt(b),\alpha^{\vee}_i\right\rangle \geq 0 \\
e^{-\left\langle \wt(b),\alpha^{\vee}_i\right\rangle}_i(b) & \text{otherwise}
\end{cases}\]

Then, if $W_{\fg} \subset GL(\Lambda)$ denotes the Weyl group of $\fg$ generated by the simple reflections $r_i$ such that $r_i(\la)=\la-\left\langle \la, \alpha^{\vee}_i\right\rangle\alpha_i$, we have:

\begin{theorem}[\cite{Kash,Kash2}] Given a reductive Lie algebra $\fg$ with Dynkin diagram $I$, for any $\fg$-crystal $B$ there exists a unique action of $W_{\fg}$ on the set $B$, $S: W_{\fg} \longrightarrow \text{Aut}(B)$, such that $S(r_i)=s_{\{i\}}$ for any $i \in I$.	
	\label{thm:kas}
\end{theorem}

The Coxeter presentation of the Weyl group is given by $W_{\fg}=\{r_i, \; i \in I \;| \; (r_ir_j)^{m_{ij}}=1\}$ where $m_{ii}=1$ and for $i \neq j$, $m_{ij}=2,3,4, \text{ or }6$ depending on whether the nodes $i$ and $j$ in the Dynkin diagram $I$ are connected by $0, 1, 2, \text{ or } 3$ edges. The relations $(r_ir_j)^{m_{ij}}=1$ for $i \neq j$ are known as the ``braid relations''. In particular, Theorem \ref{thm:kas} then means that the action of the cactus group $C_{\fg}$ factors through the quotient by the braid relations, which we will call the {\bf reduced cactus group}. 

\vspace{0.5cm}

\subsection{Another appearance of the cactus action}

For $\fg=\mgl_n$, the cactus group action from Section \ref{sec:action} is the same as the action on Gelfand--Tsetlin patterns studied by Berenstein and Kirillov in \cite{BK}. More specifically, let $K=K_n$ denote the set of all triangular arrays:
\[x=(x_{ij} \in \Z \; | \; 1 \leq i \leq j \leq n)=(x^{(n)},x^{(n-1)},\hdots,x^{(1)}) \]

\noindent where $x^{(k)}=(x_{1k},\hdots,x_{kk}) \in \Z^k$ and furthermore:
$$\begin{array}{lc	ll} x_{ij} &\geq& x_{i+1j+1}, \; &1 \leq i \leq j \leq n-1 \\
	 x_{ij} &\geq& x_{ij-1}, \; &1 \leq i \leq j \leq n \\
	  x_{ij}&\geq& 0, \; &1 \leq i \leq j \leq n.
\end{array}$$

 For any $x \in K$, denote by  $\beta(x)=(\beta_1(x),\hdots,\beta_n(x)) \subset \Z^n$ the vector given by $\beta_j(x)=|x^{(j)}|-|x^{(j-1)}|$, where $|x^{(j)}|=\sum^j_{i=1}x_{ij}$ for $j=1,\hdots,n$ and $|x^{(0)}|=0$. For a partition $\la \in \Z^n$ and vector $\beta \in \Z^n$, consider the following subsets:
\[K_{\la}=\{x \in K \; | \; x^{(n)}=\la \} \quad \quad K^{\beta}_{\la}=\{x \in K_{\la} \; | \; \beta(x)=\beta \} \]
Then $K_{\la}\cong B_{\la}$, the crystal for the irreducible, $\la$--highest weight representation of $\mgl_n$  (see \cite{GZ85}, \cite{BK}). One way to see this is using tableaux of shape $\la$, whose crystal structure is described in Remark \ref{rem:crtableau}. The bijection is then realized by mapping $x \in K_{\la}$ to a tableau $T_x$ of shape $\la=x^{(n)}$ and filling determined consecutively by the partitions $\emptyset =x^{(0)} \subset x^{(1)} \subset x^{(2)} \subset \hdots \subset x^{(n)}=\la$. Namely, the boxes in $x^{(j)}\setminus x^{(j-1)}$ are labelled by $j$.

\begin{example}\label{ex:GZT} An example of the bijection $K_{\la} \cong B_{\la}$ for $\mgl_4$ where the partition is $\la=(5,3,3,1)$.
	
	\[x=\arraycolsep=3pt\def\arraystretch{0.9}\begin{array}{ccccccc}
	5 & & 3 & & 3 & & 1 \\
	& 4 & & 3 & & 1 & \\
	& & 4 & & 2 & & \\
	& & & 3 & & & 
	\end{array}	
	\hspace{0.3cm}\mapsto\hspace{0.3cm}
	T_x=\begin{minipage}[t]{0.2\textwidth}\vspace{-1.2cm}\begin{ytableau}
	1 & 1 & 1 & 2 & 4 \\
	2 & 2 & 3 \\
	3 & 4 & 4 \\
	4
	\end{ytableau}\end{minipage}\]
\end{example}

Berenstein and Kirillov define the following elementary transformations $t_j, \; 1 \leq j \leq n-1$, on triangular arrays by, for any $x=(x_{ik} \; | \; 1 \leq i \leq k \leq n)$: 
\[t_j(x_{ik})=\begin{cases}
x_{ik} & \text{ if } k \neq j \\
\min(x_{ik+1},x_{i-1k-1}) + \max(x_{i+1k+1},x_{ik-1})-x_{ik} & \text{ if } k=j
\end{cases}\]

\begin{proposition}[\cite{BK}, Proposition 1.2]
	For any $1 \leq j \leq n-1$,
	\[t_j: K \longrightarrow K	\quad \quad t_j:K_{\la} \longrightarrow K_{\la} \quad \quad t_j: K^{\beta}_{\la} \longrightarrow K^{(j \; j+1)\cdot \beta}_{\la}\]
\end{proposition}

\vspace{0.3cm}

Moreover, if $G_n$ denotes the group generated by $t_j, \; 1\leq j \leq n-1$, then we have:

\begin{theorem}[\cite{BK}]~\\
	\vspace{-0.5cm}
	\begin{enumerate}
		
		\item An alternative set of generators for $G_n$ is $q_i=t_1(t_2t_1)\hdots t_it_{i-1} \hdots t_1$ for $1 \leq i \leq n-1$.
		
		\item For a tableau $T \in B_{\la}\cong K_{\la}$, the involution $q_{i-1}$ acts nontrivially only on $T_{\leq i}$ (the subtableau with boxes labelled by at most $i$) and its action coincides with that of the partial Sch\"{u}tzenberger involution on tableaux. 
	\end{enumerate}
\end{theorem}

This implies in particular that the action of $q_i \in G_n$ on Gelfand-Tsetlin patterns $K_{\la} \cong B_{\la}$ agrees with the action of the element $s_{\{1,2,\hdots,i\}} \in C_{\mgl_n}$ corresponding to the Dynkin subdiagram with nodes $1$ to $i$, using the standard numbering. These elements in fact generate the whole cactus group, as discussed in the proof of Theorem \ref{thm:agree} (there we use the shorter notation $s_{1,i}:=s_{\{1,2,\hdots,i-1\}}$, and $s_{i,j}:=s_{\{i,i+1,\hdots,j-1\}}$).

\begin{example}Consider $q_2=t_1t_2t_1 \in G_4$. We want to show that it acts the same way on the Gelfand-Tsetlin pattern $x$ from Example \ref{ex:GZT} as $s_{\{1,2\}} \in C_{\mgl_4}$ does on the corresponding tableau $T_x$. Starting from $x$ and applying $q_2$, we get:
	\[\arraycolsep=2pt\def\arraystretch{0.7}\begin{array}{ccccccc}
	5 & & 3 & & 3 & & 1 \\
	& 4 & & 3 & & 1 & \\
	& & 4 & & 2 & & \\
	& & & 3 & & & 
	\end{array}\xrightarrow{t_1}
	\begin{array}{ccccccc}
	5 & & 3 & & 3 & & 1 \\
	& 4 & & 3 & & 1 & \\
	& & 4 & & 2 & & \\
	& & & 3 & & & 
	\end{array}\xrightarrow{t_2}
	\begin{array}{ccccccc}
	5 & & 3 & & 3 & & 1 \\
	& 4 & & 3 & & 1 & \\
	& & 3 & & 2 & & \\
	& & & 3 & & & 
	\end{array}\xrightarrow{t_1}
	\begin{array}{ccccccc}
	5 & & 3 & & 3 & & 1 \\
	& 4 & & 3 & & 1 & \\
	& & 3 & & 2 & & \\
	& & & 2 & & & 
	\end{array}\]
	On the other hand, the action of $s_{\{1,2\}}$ only affects $(T_x)_{\leq 3}$, so we will study how that subtableau changes:
	
	\[(T_x)_{\leq 3}=\begin{minipage}[t]{0.2\textwidth}\vspace{-0.6cm}\ytableausetup{nosmalltableaux}\begin{ytableau}
	1 & 1 & 1 & 2 \\
	2 & 2 & 3 \\
	3
	\end{ytableau}\end{minipage}=f_1f_2\; \begin{minipage}[t]{0.2\textwidth}\vspace{-0.6cm}\begin{ytableau}
	1 & 1 & 1 & 1 \\
	2 & 2 & 2 \\
	3
	\end{ytableau}\end{minipage}\]
	\vspace{0.3cm}
	\[\Longrightarrow \; s_{\{1,2\}}((T_x)_{\leq 3})=e_2e_1\; \begin{minipage}[t]{0.2\textwidth}\vspace{-0.6cm}\begin{ytableau}
	1 & 2 & 2 & 3 \\
	2 & 3 & 3 \\
	3
	\end{ytableau}\end{minipage}=
	\begin{minipage}[t]{0.2\textwidth}\vspace{-0.6cm}\begin{ytableau}
	1 & 1 & 2 & 3 \\
	2 & 2 & 3 \\
	3
	\end{ytableau}\end{minipage}\]
	
	\vspace{0.2cm}
	
	Therefore, since the remaining part of the tableau does not change under the action of $s_{\{1,2\}}$:
	
		\vspace{0.2cm}
		
	\[s_{\{1,2\}}(T_x)=\begin{minipage}[t]{0.2\textwidth}\vspace{-1.2cm}\ytableausetup{nosmalltableaux}\begin{ytableau}
	1 & 1 & 2 & 3 & 4 \\
	2 & 2 & 3 \\
	3 & 4 & 4 \\
	4
	\end{ytableau}\end{minipage}\]
	
	This is precisely the tableau we obtain from $q_2(x)$ under the bijection above.
\end{example}

\begin{remark}
	In \cite{CGP16}, the authors use semistandard growth diagrams to relate the cactus group and the Berenstein--Kirillov group, and find implications between the relations on each side. 
	
	The above discussion together with Theorem \ref{prop:action} in particular implies one of the main theorems in \cite{CGP16}:
	\begin{theorem}[\cite{CGP16}, Theorem 1.4]
		The map $C_n \rightarrow G_n$ given by $s_{i,j} \mapsto q_{j-1}q_{j-i}q_{j-1}$ defines a group homomorphism.
	\end{theorem}
\end{remark}

In the next section we will focus our attention on $\mgl_k$--crystals and their tensor products.

\section{Tensor Products of $\mgl_k$--crystals}
\label{sec:tensor}
Let us first generalize the tensor product rule for $\fg$--crystals in Section \ref{sec:prelim} from two to any number of crystals.
\subsection{General tensor product rule}\label{sub:gtp}~\\
If $B_1, \hdots, B_m$ are $\fg$--crystals, then so is $B_1 \otimes \hdots \otimes B_m$ with the following operations (see \cite{Jo}): 
\vspace{0.5cm}

\begin{enumerate}
	\itemsep=2em
	\item 
	$\begin{aligned}[t]
	\mathrm{wt}(b_1 \otimes \hdots \otimes b_m) = \mathrm{wt}(b_1) + \hdots + \mathrm{wt}(b_m)
	\end{aligned}$
	\item
	$\begin{aligned}[t]
	\epsilon_i(b_1 \otimes \hdots \otimes b_m) &= \mathrm{max}_k\{\epsilon^k_i(b_1 \otimes \hdots \otimes b_m),0\}, \text{where} \\
	& \\
	\epsilon^k_i(b_1 \otimes \hdots \otimes b_m) &= \epsilon_i(b_k) - \left\langle \alpha^{\vee}_i, \mathrm{wt}(b_1) + \hdots + \mathrm{wt}(b_{k-1}) \right\rangle
	\end{aligned}$
	\item
	$\begin{aligned}[t]
	\phi_i(b_1 \otimes \hdots \otimes b_m) &= \mathrm{max}_k\{\phi^k_i(b_1 \otimes \hdots \otimes b_m),0\}, \text{where} \\
	& \\
	\phi^k_i(b_1 \otimes \hdots \otimes b_m) &= \phi_i(b_k) + \left\langle \alpha^{\vee}_i, \mathrm{wt}(b_{k+1}) + \hdots + \mathrm{wt}(b_{m}) \right\rangle
	\end{aligned}$
	\item 
	$\begin{aligned}[t]
	e_i(b_1 \otimes \hdots \otimes b_m) &= b_1 \otimes \hdots \otimes e_i(b_s) \otimes \hdots \otimes b_m, \text{where} \\
	& \\
	s \in \{1, \hdots, m\} & \text{ is the smallest number such that } \epsilon^s_i=\epsilon_i
	\end{aligned}$
	\item
	$\begin{aligned}[t]
	f_i(b_1 \otimes \hdots \otimes b_m) &= b_1 \otimes \hdots \otimes f_i(b_t) \otimes \hdots \otimes b_m, \text{where} \\
	& \\
	t \in \{1, \hdots, m\} & \text{ is the largest number such that } \phi^t_i=\phi_i
	\end{aligned}$
\end{enumerate}

\vspace{0.3cm}

\subsection{The $\mgl_n \times \mgl_m$--crystal $\Lambda^NB_{n,m}$}~\\
Consider $\Lambda^N B_{n,m}$, the set of $n \times m$ matrices with entries in $\{0,1\}$ and precisely $N$ $1$'s, which is the underlying set for the crystal of $\Lambda^N(\C^n \otimes \C^m)$ both as a $\mgl_n$-- and as a $\mgl_m$--representation.

Indeed, we have both a $\mgl_n$-- and a $\mgl_m$--crystal structure on $\Lambda^N B_{n,m}$ coming from the decompositions:
\begin{equation}
\label{eqn:col}
\begin{aligned}
\Lambda^N B_{n,m} \cong& \bigsqcup_{\underline{k} \in \mathbb{N}^m, |\underline{k}|=N} \mathcal{B}^{\mgl_n}_{\varpi_{k_1}} \otimes \hdots \otimes \mathcal{B}^{\mgl_n}_{\varpi_{k_m}} \\
\left[\begin{array}{ccc}
| & & | \\
c_1 & \hdots & c_m \\
| & & | 
\end{array}\right] &\mapsto c_m \otimes \hdots \otimes c_1
\end{aligned}
\end{equation}

\noindent where in Equation (\ref{eqn:col}) above, $\mathcal{B}^{\mgl_n}_{\varpi_i}$ is the crystal for the fundamental $\mgl_n$--representation $\Lambda^i(\C^n)$ corresponding to the fundamental weight $\varpi_i$, and in Equation (\ref{eqn:row}) below $\mathcal{B}^{\mgl_m}_{\varpi_j}$ is the crystal for the $\mgl_m$--representation $\Lambda^j(\C^m)$:

\begin{equation}
\label{eqn:row}
\begin{aligned}
\Lambda^N B_{n,m} \cong& \bigsqcup_{\underline{l} \in \mathbb{N}^n, |\underline{l}|=N} \mathcal{B}^{\mgl_m}_{\varpi_{l_1}} \otimes \hdots \otimes \mathcal{B}^{\mgl_m}_{\varpi_{l_n}} \\
\left[\begin{array}{ccc}
\textrm{\bf{---}} & r_1 & \textrm{\bf{---}} \\
& \vdots & \\
\textrm{\bf{---}} & r_n & \textrm{\bf{---}}
\end{array}\right] &\mapsto r_1 \otimes \hdots \otimes r_n
\end{aligned}
\end{equation}



We will denote the $\mgl_n$--crystal maps (which act on the columns of the matrices) by $C\ep_i$, $C\phi_i$, $Ce_i$, $Cf_i$ and $C\text{wt}$, and the $\mgl_m$--crystal maps by $R\ep_i, R\phi_i, Re_i, Rf_i$ and $R\text{wt}$, all coming from the rule in Section \ref{sub:gtp}.

\begin{proposition}[See also \cite{vL}, Lemma 1.3.7 and \cite{DK}, Theorem 2]
	\label{prop:mor}
	The $\mgl_n$-- and $\mgl_m$--crystal structures of $\Lambda^N B_{n,m}$ commute, i.e. the Kashiwara operators $Re_i$ and $Rf_i$, $i \in \{1,\hdots,m-1\}$,  are $\mgl_n$--crystal morphisms and analogously $Ce_j$ and $Cf_j, \; j \in \{1,\hdots,n-1\}$ are $\mgl_m$--crystal morphisms. This gives $\Lambda^N B_{n,m}$ the structure of a $\mgl_n \times \mgl_m$--crystal. 
\end{proposition}

\begin{proof}
	It suffices to check the following for each $M \in \Lambda^N B_{n,m}$ such that, in the left column of equations below we assume $Re_i(M) \neq 0$, in the right column $Rf_i(M) \neq 0$, and additionally $Ce_j(M)\neq 0$ or $Cf_j(M) \neq 0$ if $Ce_j$, respectively $Cf_j$ appears in the equation:
	\begin{align}
		C\wt Re_i &= C\wt & C\wt Rf_i &= C\wt \label{eqn:wt} \\
		C\epsilon_j Re_i &= C\epsilon_j & C\epsilon_j Rf_i &= C\epsilon_j \label{eqn:ep} \\
		C\phi_j Re_i &= C\phi_j & C\phi_j Rf_i &= C\phi_j \label{eqn:phi} \\
		Ce_j Re_i &=Re_i Ce_j  &  Ce_j Rf_i &= Rf_i Ce_j \label{eqn:rece} \\
		Cf_j Re_i &= Re_i Cf_j  &  Cf_j Rf_i &= Rf_i Cf_j \label{eqn:recf}
	\end{align}
	
	We also need to include Equations (\ref{eqn:wt}), (\ref{eqn:ep}), and (\ref{eqn:phi}) with the roles of ``R'' and ``C'' switched. Those verifications are analogous and we will omit them, with the exception of the analogue of (\ref{eqn:ep}), which we show in Lemma \ref{lem:cerep}.
	
	Note that in each line, once we know the equation for $Re_i$, the corresponding one for $Rf_i$ follows from:
	$$Rf_i(M)=M' \neq 0 \; \Rightarrow \; M=Re_i(M')$$
	For instance, if the first part of Equation (\ref{eqn:wt}) holds and $Rf_i(M)=M' \neq 0$, then
	$C\wt(Rf_i(M))=C\wt(M')=C\wt(Re_i(M'))=C\wt(M)$. Similarly, using that $Ce_j$ and $Cf_j$ are partial inverses, Equations (\ref{eqn:ep}) and (\ref{eqn:rece}) imply (\ref{eqn:recf}). Indeed, by definition of the Kashiwara operators:
	$$Cf_j(M)=M''\neq 0 \; \Rightarrow M=Ce_j(M'')$$
	We also assume that $Re_i(M)\neq 0$. This implies that $Re_i(M'') \neq 0$, since $R\ep_i(M) \neq 0$ and $Cf_j$ preserves $R\ep_i$ by the analogue of Equation (\ref{eqn:ep}) with $R$ and $C$ switched, so $R\ep_i(M'')=R\ep_i(Cf_j(M))=R\ep_i(M) \neq 0$. Then we can apply Equation (\ref{eqn:rece}) for $M''$ to get:
	\begin{align*}Cf_jRe_i(M)&=Cf_jRe_i(Ce_j(M''))\stackrel{(\ref{eqn:rece})}{=}Cf_jCe_j(Re_i(M'')) \\
		&=Re_i(M'')=Re_iCf_j(M)
	\end{align*}
	
	Furthermore, Equations (\ref{eqn:wt}) and (\ref{eqn:ep}) imply (\ref{eqn:phi}). Indeed, for $Re_i(M) \neq 0$:
	\begin{align*}C\phi_j(Re_i(M)) &=C\ep_j(Re_i(M))+\left\langle C\wt(Re_i(M)),C\alpha^{\vee}_j \right\rangle \\
		&=C\ep_j(M)+\left\langle C\wt(M),C\alpha^{\vee}_j \right\rangle = C\phi_j(M)
	\end{align*}
	
	We will break the proofs of Equations (\ref{eqn:wt}), (\ref{eqn:ep}), and (\ref{eqn:rece}) into several lemmas below. 
\end{proof}

\begin{lemma}
	The operator $Re_i$, $i \in \{1,\hdots,m-1\}$, preserves the $\mgl_n$--weight of each element:
	\[ \forall \; M \in \Lambda^N B_{n,m}, \;  Re_i(M) \neq 0 \Rightarrow C\emph{wt}(Re_i(M)) = C\emph{wt}(M)\]
\end{lemma}

\begin{proof}Given a matrix $M$, if $Re_i$ acts in the $l$--th row and we set $\widetilde{M}=Re_i(M)$, then $\widetilde{M}_{pq}=M_{pq}$ when $(p,q) \neq (l,i),(l,i+1)$ and $\widetilde{M}_{l,i}=M_{l,i}+1, \; \widetilde{M}_{l,i+1}=M_{l,i+1}-1$.
	Note that the $\mgl_n$--weight of $M$ is given by $C\wt(M)=\left(\sum^m_{s=1}M_{1,s},\sum^m_{s=1}M_{2,s},\hdots, \sum^m_{s=1}M_{n,s}\right)$, and so:
	\[C\wt(\widetilde{M})=\left(\sum^m_{s=1}M_{1,s},\sum^m_{s=1}M_{2,s},\hdots, \sum^m_{s=1}M_{l,s}+1-1,\sum^m_{s=1}M_{n,s}\right)=C\wt(M)\]
	
\end{proof}

\begin{lemma}
	\label{lem:recep}
	$Re_i$ preserves the value of $C\ep_j$ for all $i \in \{1,\hdots,m-1\}, \; j \in \{1, \hdots, n-1\}$:
	\[ \forall \; M \in \Lambda^N B_{n,m}, \; Re_i(M) \neq 0 \Rightarrow C\ep_j(Re_i(M)) = C\ep_j(M)\]
\end{lemma}

\begin{proof}

	Given a matrix $M$ in $\Lambda^N B_{n,m}$, suppose $Re_i$ acts in the $l$--th row:

	$$M = \begin{blockarray}{cccccc}
	\phantom{0} & \phantom{0} & \phantom{0} & \text{\tiny{i}} & \text{\tiny{i+1}} & \phantom{0} \\
	\begin{block}{c[ccccc]}
	\phantom{0} & \phantom{0} & \phantom{0} & \phantom{0} & \phantom{0} & \phantom{0} \\
	\phantom{0} & \phantom{0} & \phantom{0} & \phantom{0} & \phantom{0} & \phantom{0} \\
	\text{\tiny{l}} & \phantom{0} & \phantom{0} & 0 & 1 & \phantom{0} \\
	\phantom{0} & \phantom{0} & \phantom{0} & \phantom{0} & \phantom{0} & \phantom{0} \\
	\phantom{0} & \phantom{0} & \phantom{0} & \phantom{0} & \phantom{0} & \phantom{0} \\
	\end{block}
	\end{blockarray}
	\quad \stackrel{Re_i}{\xrightarrow{\hspace*{1cm}}}
	\begin{blockarray}{cccccc}
	\phantom{0} & \phantom{0} & \phantom{0} & \text{\tiny{i}} & \text{\tiny{i+1}} & \phantom{0} \\
	\begin{block}{c[ccccc]}
	\phantom{0} & \phantom{0} & \phantom{0} & \phantom{0} & \phantom{0} & \phantom{0} \\
	\phantom{0} & \phantom{0} & \phantom{0} & \phantom{0} & \phantom{0} & \phantom{0} \\
	\text{\tiny{l}} & \phantom{0} & \phantom{0} & 1 & 0 & \phantom{0} \\
	\phantom{0} & \phantom{0} & \phantom{0} & \phantom{0} & \phantom{0} & \phantom{0} \\
	\phantom{0} & \phantom{0} & \phantom{0} & \phantom{0} & \phantom{0} & \phantom{0} \\
	\end{block}
	\end{blockarray}$$
	
	Note that from the description of $C\ep_j$ above, we have, for $\widetilde{k}:=m+1-k$ (for simplicity of notation):
	
	\begin{align*}
	C\ep^{\widetilde{k}}_j(M) &=C\ep^{\widetilde{k}}_j\left(\left[\begin{array}{ccc}
	| & & | \\
	c_1 & \hdots & c_m \\
	| & & | 
	\end{array}\right]\right) = C\ep^{\widetilde{k}}_j(c_m \otimes c_{m-1} \otimes \hdots \otimes c_1) \\
	\vspace{0.5cm}
	&= C\ep_j(c_k) - \left\langle \alpha^{\vee}_j, \text{wt}(c_m)+\hdots + \text{wt}(c_{k+1}) \right\rangle  \\
	&= \delta_{M_{j,k},0}\delta_{M_{j+1,k},1} + \DS{\sum^{m}_{s=k+1}}{M_{j+1,s}} - \DS{\sum^{m}_{s=k+1}} {M_{j,s}} 
	\end{align*}
	
	Diagrammatically, this can be represented using the matrix $M$ as:
	
	\begin{figure}[H]
		\begin{center}
			\begin{tikzpicture}
			\matrix[matrix of math nodes,left delimiter={[}, right delimiter={]}](m)
			{
				\phantom{0} &  \phantom{0} & \phantom{0} & \phantom{0}  & \phantom{0} \\
				\phantom{0} & \delta_0 & \phantom{0} & - & \phantom{0} \\
				\phantom{0} & \delta_1 & \phantom{0} & + & \phantom{0} \\
				\phantom{0} & \phantom{0} & \phantom{0} & \phantom{0} & \phantom{0} \\
			};
			\draw[color=red,fill=red!10,fill opacity =0.5] (m-2-5.north east) rectangle (m-3-2.south west); 
			\draw[color=white](m-2-5.north east) -- (m-4-5.south east);
			\draw[color=red] (m-3-5.north east) -- (m-3-1.north east);
			\node[above] at (m-1-2.north) {\tiny{k}};
			\node[left] at ($(m.west)+(-0.45,0.25)$) {\tiny{j}};
			\node[left] at ($(m.west)+(-0.35,-0.4)$) {\tiny{j+1}};
			\draw[color=red] (m-2-3.north west) -- (m-4-3.north west);
			\end{tikzpicture}
			\caption{A short-hand notation for how $C\ep^{\widetilde{k}}_j$ is computed on a given matrix.}
			\label{fig:Cep}
		\end{center}
	\end{figure}
	
	This only involves rows $j$ and $j+1$, and $C\ep_j = \max_{\widetilde{k}} C\ep^{\widetilde{k}}_j$, so $Re_i$ acting in the $l$--th row can only affect $C\ep_{l-1}$ and $C\ep_l$. It does not affect $C\ep^{\widetilde{k}}_{l-1}, C\ep^{\widetilde{k}}_l$ for $k>i+1$ and $k<i$, as illustrated by the pictures below:
			\begin{center}
			\begin{multicols}{2}
				
				\begin{tikzpicture}
				\matrix[matrix of math nodes,left delimiter={[}, right delimiter={]}](m)
				{
					\phantom{0} &  \phantom{0} & \phantom{0} & \phantom{0}  & \phantom{0} & \phantom{0} & \phantom{0} & \phantom{0} \\
					\phantom{0} & \phantom{0} & \phantom{0}  & \phantom{0} & \delta_0 & \phantom{0} & - & \phantom{0}  \\
					\phantom{0} & 0 & 1  & \phantom{0} & \delta_1 & \phantom{0} & + & \phantom{0}  \\
					\phantom{0} & \phantom{0} & \phantom{0} & \phantom{0} & \phantom{0} & \phantom{0} & \phantom{0} & \phantom{0} \\
					\phantom{0} & \phantom{0} & \phantom{0} & \phantom{0} & \phantom{0} & \phantom{0} & \phantom{0} & \phantom{0} \\
				};
				\draw[color=red,fill=red!10,fill opacity =0.5] (m-2-8.north east) rectangle (m-3-5.south west); 
				\draw[color=white](m-2-8.north east) -- (m-4-8.south east);
				\draw[color=red] (m-3-8.north east) -- (m-3-4.north east);
				\node[above] at (m-1-5.north) {\tiny{k}};
				\node[above] at (m-1-2.north) {\tiny{i}};
				\node[above] at (m-1-3.north) {\tiny{i+1}};
				\node[left] at ($(m.west)+(-0.35,0.5)$) {\tiny{l-1}};
				\node[left] at ($(m.west)+(-0.35,-0.1)$) {\tiny{l}};
				\draw[color=red] (m-2-6.north west) -- (m-4-6.north west);
				\end{tikzpicture}
				
				\begin{tikzpicture}
				\matrix[matrix of math nodes,left delimiter={[}, right delimiter={]}](m)
				{
					\phantom{0} &  \phantom{0} & \phantom{0} & \phantom{0}  & \phantom{0} & \phantom{0} & \phantom{0} & \phantom{0} \\
					\phantom{0} & \phantom{0} & \phantom{0}  & \phantom{0} & \phantom{0} & \phantom{0} & \phantom{0} & \phantom{0}  \\
					\phantom{0} & 0 & 1  & \phantom{0} & \delta_0 & \phantom{0} & - & \phantom{0}  \\
					\phantom{0} & \phantom{0} & \phantom{0} & \phantom{0} & \delta_1 & \phantom{0} & + & \phantom{0} \\
					\phantom{0} & \phantom{0} & \phantom{0} & \phantom{0} & \phantom{0} & \phantom{0} & \phantom{0} & \phantom{0} \\
				};
				\draw[color=red,fill=red!10,fill opacity =0.5] (m-3-8.north east) rectangle (m-4-5.south west); 
				\draw[color=white](m-3-8.north east) -- (m-5-8.south east);
				\draw[color=red] (m-4-8.north east) -- (m-4-4.north east);
				\node[above] at (m-1-5.north) {\tiny{k}};
				\node[above] at (m-1-2.north) {\tiny{i}};
				\node[above] at (m-1-3.north) {\tiny{i+1}};
				\node[left] at ($(m.west)+(-0.35,0.05)$) {\tiny{l}};
				\node[left] at ($(m.west)+(-0.35,-0.65)$) {\tiny{l+1}};
				\draw[color=red] (m-3-6.north west) -- (m-5-6.north west);
				\end{tikzpicture}
			\end{multicols}
			\end{center}

			\begin{center}
			\begin{multicols}{2}
				
				\begin{tikzpicture}
				\matrix[matrix of math nodes,left delimiter={[}, right delimiter={]}](m)
				{
					\phantom{0} &  \phantom{0} & \phantom{0} & \phantom{0}  & \phantom{0} & \phantom{0} & \phantom{0} & \phantom{0} \\
					\phantom{0} & \delta_0 & \phantom{0}  & \phantom{0} & \phantom{0} & \phantom{0} & - & \phantom{0}  \\
					\phantom{0} & \delta_1 & \phantom{0}  & 0 & 1 & \phantom{0} & + & \phantom{0}  \\
					\phantom{0} & \phantom{0} & \phantom{0} & \phantom{0} & \phantom{0} & \phantom{0} & \phantom{0} & \phantom{0} \\
					\phantom{0} & \phantom{0} & \phantom{0} & \phantom{0} & \phantom{0} & \phantom{0} & \phantom{0} & \phantom{0} \\
				};
				\draw[color=red,fill=red!10,fill opacity =0.5] (m-2-8.north east) rectangle (m-3-2.south west); 
				\draw[color=white](m-2-8.north east) -- (m-4-8.south east);
				\draw[color=red] (m-3-8.north east) -- (m-3-1.north east);
				\node[above] at (m-1-2.north) {\tiny{k}};
				\node[above] at (m-1-4.north) {\tiny{i}};
				\node[above] at (m-1-5.north) {\tiny{i+1}};
				\node[left] at ($(m.west)+(-0.35,0.5)$) {\tiny{l-1}};
				\node[left] at ($(m.west)+(-0.35,-0.1)$) {\tiny{l}};
				\draw[color=red] (m-2-3.north west) -- (m-4-3.north west);
				\end{tikzpicture}

				\begin{tikzpicture}
				\matrix[matrix of math nodes,left delimiter={[}, right delimiter={]}](m)
				{
					\phantom{0} &  \phantom{0} & \phantom{0} & \phantom{0}  & \phantom{0} & \phantom{0} & \phantom{0} & \phantom{0} \\
					\phantom{0} & \phantom{0} & \phantom{0}  & \phantom{0} & \phantom{0} & \phantom{0} & \phantom{0} & \phantom{0}  \\
					\phantom{0} & \delta_0 & \phantom{0}  & 0 & 1 & \phantom{0} & - & \phantom{0}  \\
					\phantom{0} & \delta_1 & \phantom{0} & \phantom{0} & \phantom{0} & \phantom{0} & + & \phantom{0} \\
					\phantom{0} & \phantom{0} & \phantom{0} & \phantom{0} & \phantom{0} & \phantom{0} & \phantom{0} & \phantom{0} \\
				};
				\draw[color=red,fill=red!10,fill opacity =0.5] (m-3-8.north east) rectangle (m-4-2.south west); 
				\draw[color=white](m-3-8.north east) -- (m-5-8.south east);
				\draw[color=red] (m-4-8.north east) -- (m-4-1.north east);
				\node[above] at (m-1-2.north) {\tiny{k}};
				\node[above] at (m-1-4.north) {\tiny{i}};
				\node[above] at (m-1-5.north) {\tiny{i+1}};
				\node[left] at ($(m.west)+(-0.35,0.05)$) {\tiny{l}};
				\node[left] at ($(m.west)+(-0.35,-0.65)$) {\tiny{l+1}};
				\draw[color=red] (m-3-3.north west) -- (m-5-3.north west);
				\end{tikzpicture}
			\end{multicols}
		\end{center}

	So, only the values of $C\ep^{\widetilde{i}}_{l-1}, C\ep^{\widetilde{i+1}}_{l-1}, C\ep^{\widetilde{i}}_l, C\ep^{\widetilde{i+1}}_l$ can be affected by the action of $Re_i$ (the first pair in the case when $l>1$ and the second in the case $l<n$).

	First, let us look at the pair $C\ep^{\widetilde{i}}_{l-1}$, $C\ep^{\widetilde{i+1}}_{l-1}$. For simplicity, denote:
	\[A :=\sum^{m}_{s=i+2} M_{l,s} \quad \quad x:=M_{l-1,i} \quad \quad B :=\sum^{m}_{s=i+2} M_{l-1,s} \quad \quad y:=M_{l-1,i+1}\]
	Then we have, starting with the matrix $M$ on the left:
	
	$$\begin{tikzpicture}
	\matrix[matrix of math nodes,left delimiter={[}, right delimiter={]}](m)
	{
		\phantom{0} &  \phantom{0} & \phantom{0} & \phantom{0} & \phantom{0}  & \phantom{0} \\
		\phantom{0} & x & y & \phantom{0} & B & \phantom{0} \\
		\phantom{0} & 0 & 1 &\phantom{0}  & A & \phantom{0} \\
		\phantom{0} & \phantom{0} & \phantom{0} & \phantom{0} & \phantom{0} & \phantom{0} \\
	};
	\draw[color=red,fill=red!10,fill opacity =0.5] (m-2-6.north east) rectangle (m-3-2.south west); 
	\draw[color=white](m-2-6.north east) -- (m-3-6.south east);
	\draw[color=red] (m-3-2.north west) -- (m-3-6.north east);
	\draw[color=red] (m-1-2.south east) -- (m-3-2.south east);
	\node[above] at (m-1-2.north) {\tiny{i}};
	\node[above] at (m-1-3.north) {\tiny{i+1}};
	\node[left] at ($(m.west)+(-0.35,0.25)$) {\tiny{l-1}};
	\node[left] at ($(m.west)+(-0.35,-0.4)$) {\tiny{l}};
	\node[below] at ($(m.south)+(0,-0.5)$) {$C\ep^{\widetilde{i}}_{l-1}(M)=A-B-y+1$};
	\node[right] at ($(m.east)+(0.5,0)$) {$\stackrel{Re_i}{\xrightarrow{\hspace*{2cm}}}$};
	
	\begin{scope}[xshift=7cm]
	
	\matrix[matrix of math nodes,left delimiter={[}, right delimiter={]}](n)
	{
		\phantom{0} &  \phantom{0} & \phantom{0} & \phantom{0} & \phantom{0}  & \phantom{0} \\
		\phantom{0} & x & y & \phantom{0} & B & \phantom{0} \\
		\phantom{0} & 1 & 0 &\phantom{0}  & A & \phantom{0} \\
		\phantom{0} & \phantom{0} & \phantom{0} & \phantom{0} & \phantom{0} & \phantom{0} \\
	};
	\draw[color=red,fill=red!10,fill opacity =0.5] (n-2-6.north east) rectangle (n-3-2.south west); 
	\draw[color=white](n-2-6.north east) -- (n-3-6.south east);
	\draw[color=red] (n-3-2.north west) -- (n-3-6.north east);
	\draw[color=red] (n-1-2.south east) -- (n-3-2.south east);
	\node[above] at (n-1-2.north) {\tiny{i}};
	\node[above] at (n-1-3.north) {\tiny{i+1}};
	\node[left] at ($(n.west)+(-0.35,0.25)$) {\tiny{l-1}};
	\node[left] at ($(n.west)+(-0.35,-0.4)$) {\tiny{l}};
	\node[below] at ($(n.south)+(0,-0.5)$) {$C\ep^{\widetilde{i}}_{l-1}(Re_i(M))=A-B-y+\delta_{x,0}$};
	\end{scope}
	\end{tikzpicture}$$
	
	$$\hspace{-0.7cm}\begin{tikzpicture}
	\matrix[matrix of math nodes,left delimiter={[}, right delimiter={]}](m)
	{
		\phantom{0} &  \phantom{0} & \phantom{0} & \phantom{0} & \phantom{0}  & \phantom{0} \\
		\phantom{0} & x & y & \phantom{0} & B & \phantom{0} \\
		\phantom{0} & 0 & 1 &\phantom{0}  & A & \phantom{0} \\
		\phantom{0} & \phantom{0} & \phantom{0} & \phantom{0} & \phantom{0} & \phantom{0} \\
	};
	\draw[color=red,fill=red!10,fill opacity =0.5] (m-2-6.north east) rectangle (m-3-3.south west); 
	\draw[color=white](m-2-6.north east) -- (m-3-6.south east);
	\draw[color=red] (m-3-3.north west) -- (m-3-6.north east);
	\draw[color=red] (m-1-3.south east) -- (m-3-3.south east);
	\node[above] at (m-1-2.north) {\tiny{i}};
	\node[above] at (m-1-3.north) {\tiny{i+1}};
	\node[left] at ($(m.west)+(-0.35,0.25)$) {\tiny{l-1}};
	\node[left] at ($(m.west)+(-0.35,-0.4)$) {\tiny{l}};
	\node[below] at ($(m.south)+(0,-0.5)$) {$C\ep^{\widetilde{i+1}}_{l-1}(M)=A-B+\delta_{y,0}$};
	\node[right] at ($(m.east)+(0.5,0)$) {$\stackrel{Re_i}{\xrightarrow{\hspace*{2cm}}}$};
	
	\begin{scope}[xshift=7cm]
	
	\matrix[matrix of math nodes,left delimiter={[}, right delimiter={]}](n)
	{
		\phantom{0} &  \phantom{0} & \phantom{0} & \phantom{0} & \phantom{0}  & \phantom{0} \\
		\phantom{0} & x & y & \phantom{0} & B & \phantom{0} \\
		\phantom{0} & 1 & 0 &\phantom{0}  & A & \phantom{0} \\
		\phantom{0} & \phantom{0} & \phantom{0} & \phantom{0} & \phantom{0} & \phantom{0} \\
	};
	\draw[color=red,fill=red!10,fill opacity =0.5] (n-2-6.north east) rectangle (n-3-3.south west); 
	\draw[color=white](n-2-6.north east) -- (n-3-6.south east);
	\draw[color=red] (n-3-3.north west) -- (n-3-6.north east);
	\draw[color=red] (n-1-3.south east) -- (n-3-3.south east);
	\node[above] at (n-1-2.north) {\tiny{i}};
	\node[above] at (n-1-3.north) {\tiny{i+1}};
	\node[left] at ($(n.west)+(-0.35,0.25)$) {\tiny{l-1}};
	\node[left] at ($(n.west)+(-0.35,-0.4)$) {\tiny{l}};
	\node[below] at ($(n.south)+(0,-0.5)$)  {$C\ep^{\widetilde{i+1}}_{l-1}(Re_i(M))=A-B$};
	\end{scope}
	\end{tikzpicture}$$
	
	We compare next the left and right side values for the cases $y=0$ and $y=1$. Note that if $y=0$, we must also have $x=0$: If $(x,y)=(1,0)$, this would mean $R\ep^{l-2}_i=R\ep^l_i$ for $l>2$, or $R\ep^l_i=0$. That would contradict the assumption that $Re_i$ acts in row $l$, i.e. that $l$ is the smallest number such that:
	\[R\ep_i=R\ep^l_i=\max_k R\ep^k_i>0\]
	As a result,

	$$\begin{tabular}{c|c|c|c|c|}
	\cline{2-5}
	$\phantom{0}$ & \multicolumn{2}{c|}{Case 1: $y=0$} & \multicolumn{2}{|c|}{Case 2: $y=1$} \\
	\cline{2-5}
	$\phantom{0}$ & $M$ & $Re_i(M)$ & $M$ & $Re_i(M)$ \\
	\hline
	\multicolumn{1}{|c|}{$C\ep^{\widetilde{i}}_{l-1}$} & $A-B+1$ & $A-B+1$ & $A-B$ & $A-B-1+\delta_{x,0}$ \\
	\hline
	\multicolumn{1}{|c|}{$C\ep^{\widetilde{i+1}}_{l-1}$} & $A-B+1$ & $A-B$ & $A-B$ & $A-B$ \\
	\hline
	\end{tabular}$$
	
	\vspace{0.5cm}
	
	In either case, we see that $C\ep_{l-1}=\DS{\max_{\widetilde{k}} \{C\ep^{\widetilde{k}}_{l-1},0\}}$ remains unchanged.
	

	Next, we consider the pair $C\ep^{\widetilde{i}}_l$, $C\ep^{\widetilde{i+1}}_l$. In this case, we denote:
	\[A :=\sum^{m}_{s=i+2} M_{l+1,s} \quad \quad x:=M_{l+1,i} \quad \quad B :=\sum^{m}_{s=i+2} M_{l,s} \quad \quad y:=M_{l+1,i+1}\]
	As before, starting with the matrix $M$ on the left, we get:
	$$\begin{tikzpicture}
	\matrix[matrix of math nodes,left delimiter={[}, right delimiter={]}](m)
	{
		\phantom{0} &  \phantom{0} & \phantom{0} & \phantom{0} & \phantom{0}  & \phantom{0} \\
		\phantom{0} & 0 & 1 & \phantom{0} & B & \phantom{0} \\
		\phantom{0} & x & y &\phantom{0}  & A & \phantom{0} \\
		\phantom{0} & \phantom{0} & \phantom{0} & \phantom{0} & \phantom{0} & \phantom{0} \\
	};
	\draw[color=red,fill=red!10,fill opacity =0.5] (m-2-6.north east) rectangle (m-3-2.south west); 
	\draw[color=white](m-2-6.north east) -- (m-3-6.south east);
	\draw[color=red] (m-2-2.south west) -- (m-3-6.north east);
	\draw[color=red] (m-1-2.south east) -- (m-3-2.south east);
	\node[above] at (m-1-2.north) {\tiny{i}};
	\node[above] at (m-1-3.north) {\tiny{i+1}};
	\node[left] at ($(m.west)+(-0.35,0.25)$) {\tiny{l}};
	\node[left] at ($(m.west)+(-0.35,-0.4)$) {\tiny{l+1}};
	\node[below] at ($(m.south)+(0,-0.5)$) {$C\ep^{\widetilde{i}}_l(M)=A-B+y-1+x$};
	\node[right] at ($(m.east)+(0.5,0)$) {$\stackrel{Re_i}{\xrightarrow{\hspace*{2cm}}}$};
	
	\begin{scope}[xshift=7cm]
	
	\matrix[matrix of math nodes,left delimiter={[}, right delimiter={]}](n)
	{
		\phantom{0} &  \phantom{0} & \phantom{0} & \phantom{0} & \phantom{0}  & \phantom{0} \\
		\phantom{0} & 1 & 0 & \phantom{0} & B & \phantom{0} \\
		\phantom{0} & x & y &\phantom{0}  & A & \phantom{0} \\
		\phantom{0} & \phantom{0} & \phantom{0} & \phantom{0} & \phantom{0} & \phantom{0} \\
	};
	\draw[color=red,fill=red!10,fill opacity =0.5] (n-2-6.north east) rectangle (n-3-2.south west); 
	\draw[color=white](n-2-6.north east) -- (n-3-6.south east);
	\draw[color=red] (n-2-2.south west) -- (n-3-6.north east);
	\draw[color=red] (n-1-2.south east) -- (n-3-2.south east);
	\node[above] at (n-1-2.north) {\tiny{i}};
	\node[above] at (n-1-3.north) {\tiny{i+1}};
	\node[left] at ($(n.west)+(-0.35,0.25)$) {\tiny{l}};
	\node[left] at ($(n.west)+(-0.35,-0.4)$) {\tiny{l+1}};
	\node[below] at ($(n.south)+(0,-0.5)$) {$C\ep^{\widetilde{i}}_l(Re_i(M))=A-B+y$};
	\end{scope}
	\end{tikzpicture}$$
	
	$$\hspace{0.15cm}\begin{tikzpicture}
	\matrix[matrix of math nodes,left delimiter={[}, right delimiter={]}](m)
	{
		\phantom{0} &  \phantom{0} & \phantom{0} & \phantom{0} & \phantom{0}  & \phantom{0} \\
		\phantom{0} & 0 & 1 & \phantom{0} & B & \phantom{0} \\
		\phantom{0} & x & y &\phantom{0}  & A & \phantom{0} \\
		\phantom{0} & \phantom{0} & \phantom{0} & \phantom{0} & \phantom{0} & \phantom{0} \\
	};
	\draw[color=red,fill=red!10,fill opacity =0.5] (m-2-6.north east) rectangle (m-3-3.south west); 
	\draw[color=white](m-2-6.north east) -- (m-3-6.south east);
	\draw[color=red] (m-3-3.north west) -- (m-3-6.north east);
	\draw[color=red] (m-1-3.south east) -- (m-3-3.south east);
	\node[above] at (m-1-2.north) {\tiny{i}};
	\node[above] at (m-1-3.north) {\tiny{i+1}};
	\node[left] at ($(m.west)+(-0.35,0.25)$) {\tiny{l}};
	\node[left] at ($(m.west)+(-0.35,-0.4)$) {\tiny{l+1}};
	\node[below] at ($(m.south)+(0,-0.5)$) {$C\ep^{\widetilde{i+1}}_l(M)=A-B$};
	\node[right] at ($(m.east)+(0.5,0)$) {$\stackrel{Re_i}{\xrightarrow{\hspace*{2cm}}}$};
	
	\begin{scope}[xshift=7cm]
	
	\matrix[matrix of math nodes,left delimiter={[}, right delimiter={]}](n)
	{
		\phantom{0} &  \phantom{0} & \phantom{0} & \phantom{0} & \phantom{0}  & \phantom{0} \\
		\phantom{0} & 1 & 0 & \phantom{0} & B & \phantom{0} \\
		\phantom{0} & x & y &\phantom{0}  & A & \phantom{0} \\
		\phantom{0} & \phantom{0} & \phantom{0} & \phantom{0} & \phantom{0} & \phantom{0} \\
	};
	\draw[color=red,fill=red!10,fill opacity =0.5] (n-2-6.north east) rectangle (n-3-3.south west); 
	\draw[color=white](n-2-6.north east) -- (n-3-6.south east);
	\draw[color=red] (n-3-3.north west) -- (n-3-6.north east);
	\draw[color=red] (n-1-3.south east) -- (n-3-3.south east);
	\node[above] at (n-1-2.north) {\tiny{i}};
	\node[above] at (n-1-3.north) {\tiny{i+1}};
	\node[left] at ($(n.west)+(-0.35,0.25)$) {\tiny{l}};
	\node[left] at ($(n.west)+(-0.35,-0.4)$) {\tiny{l+1}};
	\node[below] at ($(n.south)+(0,-0.5)$)  {$C\ep^{\widetilde{i+1}}_l(Re_i(M))=A-B+y$};
	\end{scope}
	\end{tikzpicture}$$
	
	Next, we compare the left and right side values in the two cases $y=0$ and $y=1$. Note that if $y=1$, we must also have $x=1$. Indeed, if $(x,y)=(0,1)$, this would mean $R\ep^{l+1}_i>R\ep^l_i$, which would contradict the assumption that $Re_i$ acts in row $l$, i.e. that $l$ is the smallest number such that:
	\[R\ep_i=R\ep^l_i=\max_k R\ep^k_i>0\]
	As a result,

	$$\begin{tabular}{c|c|c|c|c|}
	\cline{2-5}
	$\phantom{0}$ & \multicolumn{2}{c|}{Case 1: $y=0$} & \multicolumn{2}{|c|}{Case 2: $y=1$} \\
	\cline{2-5}
	$\phantom{0}$ & $M$ & $Re_i(M)$ & $M$ & $Re_i(M)$ \\
	\hline
	\multicolumn{1}{|c|}{$C\ep^{\widetilde{i}}_l$} & $A-B-1+x$ & $A-B$ & $A-B+1$ & $A-B+1$ \\
	\hline
	\multicolumn{1}{|c|}{$C\ep^{\widetilde{i+1}}_l$} & $A-B$ & $A-B$ & $A-B$ & $A-B+1$ \\
	\hline
	\end{tabular}$$
	
	\vspace{0.5cm}
	
	In either case, we see that $C\ep_l=\DS{\max_{\widetilde{k}} \{C\ep^{\widetilde{k}}_l,0\}}$ remains unchanged.
\end{proof}

\begin{lemma}
	\label{lem:cerep}
	$Ce_i$ preserves the value of $R\ep_j$  for all $i \in \{1,\hdots,n-1\}$ and $j \in \{1, \hdots, m-1\}$:
	\[ \forall \; M \in \Lambda^N B_{n,m}, \; Ce_i(M) \neq 0 \;  \Rightarrow \; R\ep_j(Ce_i(M)) = R\ep_j(M)\]
\end{lemma}

\begin{proof}

	Given a matrix $M$ in $\Lambda^N B_{n,m}$, suppose $Ce_i$ acts in the $l$--th column:

	$$M = \begin{blockarray}{cccccc}
	\phantom{0} & \phantom{0} & \phantom{0} & \phantom{0} & \text{\tiny{l}} & \phantom{0} \\
	\begin{block}{c[ccccc]}
	\phantom{0} & \phantom{0} & \phantom{0} & \phantom{0} & \phantom{0} & \phantom{0} \\
	\phantom{0} & \phantom{0} & \phantom{0} & \phantom{0} & \phantom{0} & \phantom{0} \\
	\text{\tiny{i}} & \phantom{0} & \phantom{0} & \phantom{0} & 0 & \phantom{0} \\
	\text{\tiny{i+1}} & \phantom{0} & \phantom{0} & \phantom{0} & 1 & \phantom{0} \\
	\phantom{0} & \phantom{0} & \phantom{0} & \phantom{0} & \phantom{0} & \phantom{0} \\
	\end{block}
	\end{blockarray}
	\quad \stackrel{Ce_i}{\xrightarrow{\hspace*{1cm}}}
	\begin{blockarray}{cccccc}
	\phantom{0} & \phantom{0} & \phantom{0} & \phantom{0} & \text{\tiny{l}} & \phantom{0} \\
	\begin{block}{c[ccccc]}
	\phantom{0} & \phantom{0} & \phantom{0} & \phantom{0} & \phantom{0} & \phantom{0} \\
	\phantom{0} & \phantom{0} & \phantom{0} & \phantom{0} & \phantom{0} & \phantom{0} \\
	\text{\tiny{i}} & \phantom{0} & \phantom{0} & \phantom{0} & 1 & \phantom{0} \\
	\text{\tiny{i+1}} & \phantom{0} & \phantom{0} & \phantom{0} & 0 & \phantom{0} \\
	\phantom{0} & \phantom{0} & \phantom{0} & \phantom{0} & \phantom{0} & \phantom{0} \\
	\end{block}
	\end{blockarray}$$
	
	From the definition of $R\ep_j$, we have:
	
	\begin{align*}
	R\ep^k_j(M) &=R\ep^k_j\left(\left[\begin{array}{ccc}
	\textrm{\bf{---}} & r_1 & \textrm{\bf{---}} \\
	& \vdots & \\
	\textrm{\bf{---}} & r_n & \textrm{\bf{---}}
	\end{array}\right] \right)=R\ep^k_j(r_1 \otimes r_2 \otimes \hdots \otimes r_n) \\
	&= R\ep_j(r_k) - \left\langle \alpha^{\vee}_j, \text{wt}(r_1)+\hdots + \text{wt}(r_{k-1}) \right\rangle \\
	&= \delta_{M_{k,j},0}\delta_{M_{k,j+1},1} + \DS{\sum^{k-1}_{s=1}}{M_{s,j+1}} - \DS{\sum^{k-1}_{s=1}}{M_{s,j}}
	\end{align*}
	
	Diagrammatically, this can be represented within the matrix $M$ as:
	
	\begin{figure}[H]
		$$\begin{tikzpicture}
		\matrix[matrix of math nodes,left delimiter={[}, right delimiter={]}](m)
		{
			\phantom{0} &  \phantom{0} & \phantom{0} & \phantom{0}  & \phantom{0} \\
			\phantom{0} & \phantom{0} & - & + & \phantom{0} \\
			\phantom{0} &  \phantom{0} & \delta_0 & \delta_1 & \phantom{0} \\
			\phantom{0} & \phantom{0} & \phantom{0} & \phantom{0} & \phantom{0} \\
		};
		\draw[color=red,fill=red!10,fill opacity =0.5] (m-1-4.north east) rectangle (m-3-2.south east); 
		\draw[color=white](m-1-2.north east) -- (m-1-4.north east);
		\draw[color=red] (m-1-3.north east) -- (m-4-3.north east);
		\node[above] at (m-1-3.north) {\tiny{j}};
		\node[above] at (m-1-4.north) {\tiny{j+1}};
		\node[left] at ($(m.west)+(-0.35,-0.25)$) {\tiny{k}};
		\draw[color=red] (m-2-2.south east) -- (m-2-5.south west);
		\end{tikzpicture}$$
		\label{fig:Rep}
		\caption{A short-hand notation for the value of $R\ep^k_j$ on a matrix.}
	\end{figure}
	
	This only involves columns $j$ and $j+1$, and $R\ep_j = \max_k R\ep^k_j$, so $Ce_i$ acting in the $l$--th column can only affect $R\ep_{l-1}$ and $R\ep_l$. It does not affect $R\ep^k_{l-1}, R\ep^k_l$ for $k<i$ and $k>i+1$, as illustrated by the pictures below:
	
		\begin{multicols}{2}
			\hspace{0.5cm}	
			\begin{tikzpicture}
			\hspace{-0.4cm}
			\matrix[matrix of math nodes,left delimiter={[}, right delimiter={]}](m)
			{
				\phantom{0} &  \phantom{0} &  \phantom{0} & \phantom{0} & \phantom{0} \\
				\phantom{0} &  \phantom{0} & \phantom{0} & - & + \\
				\phantom{0} &  \phantom{0} &  \phantom{0} & \delta_0 & \delta_1 \\
				\phantom{0} &  \phantom{0} & \phantom{0} & \phantom{0} & \phantom{0} \\
				\phantom{0} &  \phantom{0} & \phantom{0} & \phantom{0} & 0 \\
				\phantom{0} &  \phantom{0} & \phantom{0} & \phantom{0} & 1 \\
				\phantom{0} &  \phantom{0} & \phantom{0} & \phantom{0} & \phantom{0} \\
			};
			\draw[color=red,fill=red!10,fill opacity =0.5] (m-1-5.north east) rectangle (m-3-3.south east); 
			\draw[color=white](m-1-3.north east) -- (m-1-5.north east);
			\draw[color=red] (m-1-4.north east) -- (m-4-4.north east);
			\node[above] at (m-1-4.north) {\tiny{l-1}};
			\node[above] at (m-1-5.north) {\tiny{l}};
			\node[left] at ($(m.west)+(-0.35,+0.4)$) {\tiny{k}};
			\node[left] at ($(m.west)+(-0.35,-0.55)$) {\tiny{i}};
			\node[left] at ($(m.west)+(-0.35,-1.15)$) {\tiny{i+1}};
			\draw[color=red] (m-2-4.south west) -- (m-2-5.south east);
			\end{tikzpicture}
			
			\begin{tikzpicture}
			\hspace{0.2cm}
			\matrix[matrix of math nodes,left delimiter={[}, right delimiter={]}](m)
			{
				\phantom{0} &  \phantom{0} &  \phantom{0} & \phantom{0} & \phantom{0} \\
				\phantom{0} &  \phantom{0} & \phantom{0} & - & + \\
				\phantom{0} &  \phantom{0} &  \phantom{0} & \delta_0 & \delta_1 \\
				\phantom{0} &  \phantom{0} & \phantom{0} & \phantom{0} & \phantom{0} \\
				\phantom{0} &  \phantom{0} & \phantom{0} & 0 & \phantom{0} \\
				\phantom{0} &  \phantom{0} & \phantom{0} & 1 &  \phantom{0} \\
				\phantom{0} &  \phantom{0} & \phantom{0} & \phantom{0} & \phantom{0} \\
			};
			\draw[color=red,fill=red!10,fill opacity =0.5] (m-1-5.north east) rectangle (m-3-3.south east); 
			\draw[color=white](m-1-3.north east) -- (m-1-5.north east);
			\draw[color=red] (m-1-4.north east) -- (m-4-4.north east);
			\node[above] at (m-1-4.north) {\tiny{l}};
			\node[above] at (m-1-5.north) {\tiny{l+1}};
			\node[left] at ($(m.west)+(-0.35,+0.4)$) {\tiny{k}};
			\node[left] at ($(m.west)+(-0.35,-0.55)$) {\tiny{i}};
			\node[left] at ($(m.west)+(-0.35,-1.15)$) {\tiny{i+1}};
			\draw[color=red] (m-2-4.south west) -- (m-2-5.south east);
			\end{tikzpicture}
			\end{multicols}
			
			\begin{multicols}{2}
			\begin{tikzpicture}
			\hspace{0.2cm}
			\matrix[matrix of math nodes,left delimiter={[}, right delimiter={]}](m)
			{
				\phantom{0} &  \phantom{0} &  \phantom{0} & \phantom{0} & \phantom{0} \\
				\phantom{0} &  \phantom{0} & \phantom{0} & \phantom{0} & 0 \\
				\phantom{0} &  \phantom{0} & \phantom{0} & \phantom{0} & 1 \\
				\phantom{0} &  \phantom{0} & \phantom{0} & \phantom{0} & \phantom{0} \\
				\phantom{0} &  \phantom{0} & \phantom{0} & - & + \\
				\phantom{0} &  \phantom{0} &  \phantom{0} & \delta_0 & \delta_1 \\
				\phantom{0} &  \phantom{0} & \phantom{0} & \phantom{0} & \phantom{0} \\
			};
			\draw[color=red,fill=red!10,fill opacity =0.5] (m-1-5.north east) rectangle (m-6-3.south east); 
			\draw[color=white](m-1-3.north east) -- (m-1-5.north east);
			\draw[color=red] (m-1-4.north east) -- (m-7-4.north east);
			\node[above] at (m-1-4.north) {\tiny{l-1}};
			\node[above] at (m-1-5.north) {\tiny{l}};
			\node[left] at ($(m.west)+(-0.35,+1.1)$) {\tiny{i}};
			\node[left] at ($(m.west)+(-0.35,+0.55)$) {\tiny{i+1}};
			\node[left] at ($(m.west)+(-0.35,-1.05)$) {\tiny{k}};
			\draw[color=red] (m-5-4.south west) -- (m-5-5.south east);
			\end{tikzpicture}
			\begin{tikzpicture}
			\hspace{2.45cm}
			\matrix[matrix of math nodes,left delimiter={[}, right delimiter={]}](m)
			{
				\phantom{0} &  \phantom{0} &  \phantom{0} & \phantom{0} & \phantom{0} \\
				\phantom{0} &  \phantom{0} & \phantom{0} & 0 & \phantom{0} \\
				\phantom{0} &  \phantom{0} & \phantom{0} & 1 &  \phantom{0} \\
				\phantom{0} &  \phantom{0} & \phantom{0} & \phantom{0} & \phantom{0} \\
				\phantom{0} &  \phantom{0} & \phantom{0} & - & + \\
				\phantom{0} &  \phantom{0} &  \phantom{0} & \delta_0 & \delta_1 \\
				\phantom{0} &  \phantom{0} & \phantom{0} & \phantom{0} & \phantom{0} \\
			};
			\draw[color=red,fill=red!10,fill opacity =0.5] (m-1-5.north east) rectangle (m-6-3.south east); 
			\draw[color=white](m-1-3.north east) -- (m-1-5.north east);
			\draw[color=red] (m-1-4.north east) -- (m-7-4.north east);
			\node[above] at (m-1-4.north) {\tiny{l}};
			\node[above] at (m-1-5.north) {\tiny{l+1}};
			\node[left] at ($(m.west)+(-0.35,+1.1)$) {\tiny{i}};
			\node[left] at ($(m.west)+(-0.35,+0.55)$) {\tiny{i+1}};
			\node[left] at ($(m.west)+(-0.35,-1.05)$) {\tiny{k}};
			\draw[color=red] (m-5-4.south west) -- (m-5-5.south east);
			\end{tikzpicture}
		\end{multicols}

	So, only the values of $R\ep^i_{l-1},\; R\ep^{i+1}_{l-1},\; R\ep^{i}_l,\; R\ep^{i+1}_l$ can be affected by the action of $Ce_i$, (the first pair in the case when $l>1$, and the second when $l<m$).

	First, let us look at the pair $R\ep^i_{l-1}$, $R\ep^{i+1}_{l-1}$. For simplicity, denote:
	\[A :=\sum^{i-1}_{s=1} M_{s,l} \quad \quad x:=M_{i,l-1} \quad \quad B :=\sum^{i-1}_{s=1} M_{s,l-1} \quad \quad y:=M_{i+1,l-1}\]
	Starting with the matrix $M$ on the left, we have:
	$$\begin{tikzpicture}
	\matrix[matrix of math nodes,left delimiter={[}, right delimiter={]}](m)
	{
		\phantom{0} &  \phantom{0} & B & A  & \phantom{0} \\
		\phantom{0} & \phantom{0} & x & 0 & \phantom{0} \\
		\phantom{0} &  \phantom{0} & y & 1 & \phantom{0} \\
		\phantom{0} & \phantom{0} & \phantom{0} & \phantom{0} & \phantom{0} \\
	};
	\draw[color=red,fill=red!10,fill opacity =0.5] (m-1-4.north east) rectangle (m-2-3.south west); 
	\draw[color=white](m-1-3.north west) -- (m-1-4.north east);
	\draw[color=red] (m-1-4.north west) -- (m-3-4.north west);
	\node[above] at (m-1-3.north) {\tiny{l-1}};
	\node[above] at (m-1-4.north) {\tiny{l}};\\
	\node[left] at ($(m.west)+(-0.35,+0.25)$) {\tiny{i}};
	\node[left] at ($(m.west)+(-0.2,-0.3)$) {\tiny{i+1}};
	\draw[color=red] (m-1-3.south west) -- (m-1-5.south west);
	\node[below] at ($(m.south)+(0,-0.5)$) {$R\ep^i_{l-1}(M)=A-B$};
	\node[right] at ($(m.east)+(0.5,0)$) {$\stackrel{Ce_i}{\xrightarrow{\hspace*{2cm}}}$};
	
	\begin{scope}[xshift=6.7cm]
	
	\matrix[matrix of math nodes,left delimiter={[}, right delimiter={]}](m)
	{
		\phantom{0} &  \phantom{0} & B & A  & \phantom{0} \\
		\phantom{0} & \phantom{0} & x & 1 & \phantom{0} \\
		\phantom{0} &  \phantom{0} & y & 0 & \phantom{0} \\
		\phantom{0} & \phantom{0} & \phantom{0} & \phantom{0} & \phantom{0} \\
	};
	\draw[color=red,fill=red!10,fill opacity =0.5] (m-1-4.north east) rectangle (m-2-3.south west); 
	\draw[color=white](m-1-3.north west) -- (m-1-4.north east);
	\draw[color=red] (m-1-4.north west) -- (m-3-4.north west);
	\node[above] at (m-1-3.north) {\tiny{l-1}};
	\node[above] at (m-1-4.north) {\tiny{l}};\\
	\node[left] at ($(m.west)+(-0.35,+0.25)$) {\tiny{i}};
	\node[left] at ($(m.west)+(-0.2,-0.3)$) {\tiny{i+1}};
	\draw[color=red] (m-1-3.south west) -- (m-1-5.south west);
	\node[below] at ($(n.south)+(0,-0.5)$) {$R\ep^i_{l-1}(Ce_i(M))=A-B-x+1$};
	\end{scope}
	\end{tikzpicture}$$
	
	$$\begin{tikzpicture}
	\matrix[matrix of math nodes,left delimiter={[}, right delimiter={]}](m)
	{
		\phantom{0} &  \phantom{0} & B & A  & \phantom{0} \\
		\phantom{0} & \phantom{0} & x & 0 & \phantom{0} \\
		\phantom{0} &  \phantom{0} & y & 1 & \phantom{0} \\
		\phantom{0} & \phantom{0} & \phantom{0} & \phantom{0} & \phantom{0} \\
	};
	\draw[color=red,fill=red!10,fill opacity =0.5] (m-1-4.north east) rectangle (m-3-3.south west); 
	\draw[color=white](m-1-3.north west) -- (m-1-4.north east);
	\draw[color=red] (m-1-4.north west) -- (m-4-4.north west);
	\node[above] at (m-1-3.north) {\tiny{l-1}};
	\node[above] at (m-1-4.north) {\tiny{l}};\\
	\node[left] at ($(m.west)+(-0.35,+0.25)$) {\tiny{i}};
	\node[left] at ($(m.west)+(-0.2,-0.3)$) {\tiny{i+1}};
	\draw[color=red] (m-2-3.south west) -- (m-2-5.south west);
	\node[below] at ($(m.south)+(0,-0.5)$) {$R\ep^{i+1}_{l-1}(M)=A-B-x+1-y$};
	\node[right] at ($(m.east)+(0.5,0)$) {$\stackrel{Ce_i}{\xrightarrow{\hspace*{2cm}}}$};
	
	\begin{scope}[xshift=6.7cm]
	
	\matrix[matrix of math nodes,left delimiter={[}, right delimiter={]}](m)
	{
		\phantom{0} &  \phantom{0} & B & A  & \phantom{0} \\
		\phantom{0} & \phantom{0} & x & 1 & \phantom{0} \\
		\phantom{0} &  \phantom{0} & y & 0 & \phantom{0} \\
		\phantom{0} & \phantom{0} & \phantom{0} & \phantom{0} & \phantom{0} \\
	};
	\draw[color=red,fill=red!10,fill opacity =0.5] (m-1-4.north east) rectangle (m-3-3.south west); 
	\draw[color=white](m-1-3.north west) -- (m-1-4.north east);
	\draw[color=red] (m-1-4.north west) -- (m-4-4.north west);
	\node[above] at (m-1-3.north) {\tiny{l-1}};
	\node[above] at (m-1-4.north) {\tiny{l}};
	\node[left] at ($(m.west)+(-0.35,+0.25)$) {\tiny{i}};
	\node[left] at ($(m.west)+(-0.2,-0.3)$) {\tiny{i+1}};
	\draw[color=red] (m-2-3.south west) -- (m-2-5.south west);
	\node[below] at ($(n.south)+(0,-0.5)$) {$R\ep^{i+1}_{l-1}(Ce_i(M))=A-B-x+1$};
	\end{scope}
	\end{tikzpicture}$$
	
	We compare below the left and right side values in the two cases $x=0$ and $x=1$. Note that if $x=0$, we must also have $y=0$. If, on the contrary, $(x,y)=(0,1)$, this would mean $C\ep^{\widetilde{l-1}}_i>C\ep^{\widetilde{l}}_i$ (in this setting $l>1$). This contradicts the assumption that $Ce_i$ acts in column $l$, i.e. that $l$ is closest to $m$ such that:
	\[C\ep_i=C\ep^{\widetilde{l}}_i=\max_{\widetilde{k}} C\ep^{\widetilde{k}}_i>0\]
	As a result,
	
	$$\begin{tabular}{c|c|c|c|c|}
	\cline{2-5}
	$\phantom{0}$ & \multicolumn{2}{c|}{Case 1: $x=0$} & \multicolumn{2}{|c|}{Case 2: $x=1$} \\
	\cline{2-5}
	$\phantom{0}$ & $M$ & $Ce_i(M)$ & $M$ & $Ce_i(M)$ \\
	\hline
	\multicolumn{1}{|c|}{$R\ep^i_{l-1}$} & $A-B$ & $A-B+1$ & $A-B$ & $A-B$ \\
	\hline
	\multicolumn{1}{|c|}{$R\ep^{i+1}_{l-1}$} & $A-B+1$ & $A-B+1$ & $A-B-y$ & $A-B$ \\
	\hline
	\end{tabular}$$
	
	\vspace{0.5cm}
	
	In either case, we see that $R\ep_{l-1}=\max_k \{R\ep^k_{l-1},0\}$ remains unchanged, while the earliest place in the matrix that the maximum is achieved might switch from  $R\ep^{i+1}_{l-1}$ to $R\ep^i_{l-1}$.
	
	Next, we consider $R\ep^i_l$, $R\ep^{i+1}_l$. Analogously to before, denote:
	\[A :=\sum^{i-1}_{s=1} M_{s,l+1} \quad \quad x:=M_{i,l+1} \quad \quad B :=\sum^{i-1}_{s=1} M_{s,l} \quad \quad y:=M_{i+1,l+1}\]
	Starting with the matrix $M$ on the left, we get:
	$$\begin{tikzpicture}
	\matrix[matrix of math nodes,left delimiter={[}, right delimiter={]}](m)
	{
		\phantom{0} &  \phantom{0} & B & A  & \phantom{0} \\
		\phantom{0} & \phantom{0} & 0 & x & \phantom{0} \\
		\phantom{0} &  \phantom{0} & 1 & y & \phantom{0} \\
		\phantom{0} & \phantom{0} & \phantom{0} & \phantom{0} & \phantom{0} \\
	};
	\draw[color=red,fill=red!10,fill opacity =0.5] (m-1-4.north east) rectangle (m-2-3.south west); 
	\draw[color=white](m-1-3.north west) -- (m-1-4.north east);
	\draw[color=red] (m-1-4.north west) -- (m-3-4.north west);
	\node[above] at (m-1-3.north) {\tiny{l}};
	\node[above] at (m-1-4.north) {\tiny{l+1}};\\
	\node[left] at ($(m.west)+(-0.35,+0.25)$) {\tiny{i}};
	\node[left] at ($(m.west)+(-0.2,-0.3)$) {\tiny{i+1}};
	\draw[color=red] (m-1-3.south west) -- (m-1-5.south west);
	\node[below] at ($(m.south)+(0,-0.5)$) {$R\ep^i_l(M)=A-B+x$};
	\node[right] at ($(m.east)+(0.5,0)$) {$\stackrel{Ce_i}{\xrightarrow{\hspace*{2cm}}}$};
	
	\begin{scope}[xshift=6.7cm]
	
	\matrix[matrix of math nodes,left delimiter={[}, right delimiter={]}](m)
	{
		\phantom{0} &  \phantom{0} & B & A  & \phantom{0} \\
		\phantom{0} & \phantom{0} & 1 & x & \phantom{0} \\
		\phantom{0} &  \phantom{0} & 0 & y & \phantom{0} \\
		\phantom{0} & \phantom{0} & \phantom{0} & \phantom{0} & \phantom{0} \\
	};
	\draw[color=red,fill=red!10,fill opacity =0.5] (m-1-4.north east) rectangle (m-2-3.south west); 
	\draw[color=white](m-1-3.north west) -- (m-1-4.north east);
	\draw[color=red] (m-1-4.north west) -- (m-3-4.north west);
	\node[above] at (m-1-3.north) {\tiny{l}};
	\node[above] at (m-1-4.north) {\tiny{l+1}};\\
	\node[left] at ($(m.west)+(-0.35,+0.25)$) {\tiny{i}};
	\node[left] at ($(m.west)+(-0.2,-0.3)$) {\tiny{i+1}};
	\draw[color=red] (m-1-3.south west) -- (m-1-5.south west);
	\node[below] at ($(n.south)+(0,-0.5)$) {$R\ep^i_l(Ce_i(M))=A-B$};
	\end{scope}
	\end{tikzpicture}$$
	
	$$\begin{tikzpicture}
	\matrix[matrix of math nodes,left delimiter={[}, right delimiter={]}](m)
	{
		\phantom{0} &  \phantom{0} & B & A  & \phantom{0} \\
		\phantom{0} & \phantom{0} & 0 & x & \phantom{0} \\
		\phantom{0} &  \phantom{0} & 1 & y & \phantom{0} \\
		\phantom{0} & \phantom{0} & \phantom{0} & \phantom{0} & \phantom{0} \\
	};
	\draw[color=red,fill=red!10,fill opacity =0.5] (m-1-4.north east) rectangle (m-3-3.south west); 
	\draw[color=white](m-1-3.north west) -- (m-1-4.north east);
	\draw[color=red] (m-1-4.north west) -- (m-4-4.north west);
	\node[above] at (m-1-3.north) {\tiny{l}};
	\node[above] at (m-1-4.north) {\tiny{l+1}};\\
	\node[left] at ($(m.west)+(-0.35,+0.25)$) {\tiny{i}};
	\node[left] at ($(m.west)+(-0.2,-0.3)$) {\tiny{i+1}};
	\draw[color=red] (m-2-3.south west) -- (m-2-5.south west);
	\node[below] at ($(m.south)+(0,-0.5)$) {$R\ep^{i+1}_l(M)=A-B+x$};
	\node[right] at ($(m.east)+(0.5,0)$) {$\stackrel{Ce_i}{\xrightarrow{\hspace*{2cm}}}$};
	
	\begin{scope}[xshift=6.7cm]
	
	\matrix[matrix of math nodes,left delimiter={[}, right delimiter={]}](m)
	{
		\phantom{0} &  \phantom{0} & B & A  & \phantom{0} \\
		\phantom{0} & \phantom{0} & 1 & x & \phantom{0} \\
		\phantom{0} &  \phantom{0} & 0 & y & \phantom{0} \\
		\phantom{0} & \phantom{0} & \phantom{0} & \phantom{0} & \phantom{0} \\
	};
	\draw[color=red,fill=red!10,fill opacity =0.5] (m-1-4.north east) rectangle (m-3-3.south west); 
	\draw[color=white](m-1-3.north west) -- (m-1-4.north east);
	\draw[color=red] (m-1-4.north west) -- (m-4-4.north west);
	\node[above] at (m-1-3.north) {\tiny{l}};
	\node[above] at (m-1-4.north) {\tiny{l+1}};
	\node[left] at ($(m.west)+(-0.35,+0.25)$) {\tiny{i}};
	\node[left] at ($(m.west)+(-0.2,-0.3)$) {\tiny{i+1}};
	\draw[color=red] (m-2-3.south west) -- (m-2-5.south west);
	\node[below] at ($(n.south)+(0,-0.5)$) {$R\ep^{i+1}_l(Ce_i(M))=A-B+x-1+y$};
	\end{scope}
	\end{tikzpicture}$$
	
	We compute next the values of the left and right sides in the two cases $x=0$ and $x=1$.	Note that if $x=1$, we must also have $y=1$. If $(x,y)=(1,0)$, this would mean $C\ep^{\widetilde{l}}_i=C\ep^{\widetilde{l+2}}_i$ (if $l+2 \leq m$), or $C\ep^{\widetilde{l}}_i=0$. This contradicts the assumption that $Ce_i$ acts in column $l$, i.e. that $l$ is closest to $m$ such that:
	\[C\ep^{\widetilde{l}}_i=\max_{\widetilde{k}} C\ep^{\widetilde{k}}_i>0\]
	As a result,
	
	$$\begin{tabular}{c|c|c|c|c|}
	\cline{2-5}
	$\phantom{0}$ & \multicolumn{2}{c|}{Case 1: $x=0$} & \multicolumn{2}{|c|}{Case 2: $x=1$} \\
	\cline{2-5}
	$\phantom{0}$ & $M$ & $Ce_i(M)$ & $M$ & $Ce_i(M)$ \\
	\hline
	\multicolumn{1}{|c|}{$R\ep^i_l$} & $A-B$ & $A-B$ & $A-B+1$ & $A-B$ \\
	\hline
	\multicolumn{1}{|c|}{$R\ep^{i+1}_l$} & $A-B$ & $A-B-1+y$ & $A-B+1$ & $A-B+1$ \\
	\hline
	\end{tabular}$$
	
	\vspace{0.5cm}
	
	In either case, we see that $R\ep_{l}=\max_k \{R\ep^k_{l},0\}$ remains unchanged, while the place closest to $m$ where the maximum is achieved might switch from  $R\ep^i_l$ to $R\ep^{i+1}_l$ or conversely.
\end{proof}

\begin{lemma}
	\label{lem:com}
	For all  $i \in \{1,\hdots,m-1\}, \; j \in \{1, \hdots, n-1\}$:
	\[Re_i Ce_j(M) = Ce_j Re_i(M) \; \; \forall \; M \in \Lambda^N B_{n,m} \text{ s.t. } Ce_j(M),\; Re_i(M) \neq 0\]
\end{lemma}

\begin{proof}
	Given an element $M \in \Lambda^N B_{n,m}$, suppose $Re_i$ acts in the $l$--th row. Namely, $l$ is the first row counting from the top such that $R \ep_i=R\ep^l_i>0$ and $M_{l,i}=0,\; M_{l,i+1}=1$. Then the image $Re_i(M)$ has the same entries as $M$ except $Re_i(M)_{l,i}=0,\; Re_i(M)_{l,i+1}=1$:
	
	\[ 
	Re_i(M)_{p,q}=\begin{cases} 
	M_{p,q} & (p,q) \neq (l,i), \; (l,i+1) \\
	M_{l,i}+1 & (p,q)=(l,i) \\
	M_{l,i+1}-1 & (p,q)=(l,i+1) \\ 
	\end{cases}
	\]
	
	Therefore, by Lemma \ref{lem:recep}, $Re_i$ affects only $C\ep^{\widetilde{i}}_{l-1}$, $C\ep^{\widetilde{i+1}}_{l-1}$, $C\ep^{\widetilde{i}}_l$, $C\ep^{\widetilde{i+1}}_l$ and does not change the overall value of  $C\ep_s, \; \forall \; s \in \{1,\hdots,n-1\}$.
	
	Suppose that $Ce_j$ acts in the $k$-th column. Namely $k$ is the closest column to the rightmost one such that $C\ep_j=C\ep^{\widetilde{k}}_j>0$ and $M_{j,k}=0,\; M_{j+1,k}=1$. Then the image $Ce_j(M)$ has the same entries as $M$ except $Ce_j(M)_{j,k}=1,\; Ce_j(M)_{j+1,k}=0$:
	
	\[ 
	Ce_j(M)_{p,q}=\begin{cases} 
	M_{p,q} & (p,q) \neq (j,k), \; (j+1,k) \\
	M_{j,k}+1 & (p,q)=(j,k) \\
	M_{j+1,k}-1 & (p,q)=(j+1,k) \\ 
	\end{cases}
	\]
	
	Note that $Ce_j$ affects only $R\ep^j_{k-1}$, $R\ep^{j+1}_{k-1}$, $R\ep^j_k$, $R\ep^{j+1}_k$ and does not change the overall value of  $R\ep_t, \; \forall \; t \in \{1,\hdots,m-1\}$.

	So, the actions of $Re_i$ and $Ce_j$ on $M$ are determined respectively by $R\ep^l_i$ and $C\ep^{\widetilde{k}}_j$, and interact with each other if:
	\[
	R\ep^l_i \stackrel{(1)}{=} R\ep^{j,j+1}_{k,k-1} \quad \quad \quad \quad C\ep^{\widetilde{k}}_j \stackrel{(2)}{=} C\ep^{\widetilde{i},\widetilde{i+1}}_{l,l-1}
	\]
	Namely, we need to determine if the maxima of $R\ep_i$ and $C\ep_j$ at $R\ep^l_i$ and $C\ep^{\widetilde{k}}_j$ respectively get shifted.
	
	\noindent \textsl{Case 1.} $j<l-1$ or $j>l$. 
	
	In this case, neither $(1)$ nor $(2)$ can be true, so the actions of $Re_i$ and $Ce_j$ are unaffected by one another and hence commute.
	
	\noindent \textsl{Case 2.} $j=l-1$.
	
	2.1. If $i<k-1$ or $i>k$, $(1)$ and $(2)$ still cannot be true, so $Re_i$ and $Ce_j$ do not interact and hence commute.
	
	2.2. If $i=k-1$, then $R\ep^l_i\stackrel{(1)}{=}R\ep^{j+1}_{k-1}, \; C\ep^{\widetilde{k}}_j\stackrel{(2)}{=}C\ep^{\widetilde{i+1}}_{l-1}$ and the action of $Re_i$ and $Ce_j$ happens in the $2 \times 2$ submatrix with rows indexed by $j, \; j+1$ and columns $i,\; i+1$:
	\[M^{i,i+1}_{j,j+1}=
	\left[\begin{array}{cc}
	a & 0 \\
	0 & 1
	\end{array}\right]
	\]
	If $a=1$, then $R\ep^{j+1}_i=R\ep^{j-1}_i$ for $j>1$ or $R\ep^{j+1}_i=0$, which contradicts the assumption that $l=j+1$ is the smallest number for which $R\ep_i=R\ep^{j+1}_i>0$. Therefore, we must have $a=0$. Then the actions commute in the following way, the original submatrix $M^{i,i+1}_{j,j+1}$ being in the upper left corner:
	\[
	\begin{tikzcd}[ampersand replacement=\&]
	\begin{array}{cc} 0 & 0 \\ 0 & 1 \end{array} \arrow{r}{Re_i} \arrow[swap]{d}{Ce_j} \& 
	\begin{array}{cc} 0 & 0 \\ 1 & 0 \end{array} \arrow{d}{Ce_j} \\
	\begin{array}{cc} 0 & 1 \\ 0 & 0 \end{array} \arrow{r}{Re_i} \& 
	\begin{array}{cc} 1 & 0 \\ 0 & 0 \end{array} \\
	\end{tikzcd}
	\vspace{-0.75cm}
	\]
	
	2.3. If $i=k$, on one hand we must have $M_{j+1,i}=0,\; M_{j+1,i+1}=1$ for $Re_i$ to act in the $l=(j+1)$--st row. On the other hand, we must have $M_{j,i}=0,\; M_{j+1,i}=1$ for $Ce_j$ to act in the $i=k$--th column. This leads to a contradiction, so $i=k$ is not possible.

	\noindent \textsl{Case 3.} $j=l$.
	
	3.1. If $i<k-1$ or $i>k$, $(1)$ and $(2)$ still cannot be true, so $Re_i$ and $Ce_j$ do not interact and hence commute.
	
	3.2. If $i=k-1$, then on one hand we must have $M_{j,i}=0,\; M_{j,i+1}=1$ for $Re_i$ to act in the $(l=j)$--th row. On the other hand, we must have $M_{j,i+1}=0,\; M_{j+1,i+1}=1$ for $Ce_j$ to act in the $k=(i+1)$--st column. This leads to a contradiction, so $i=k-1$ is not possible.

	3.3. If $i=k$, $R\ep^l_i\stackrel{(1)}{=}R\ep^j_k, \; C\ep^{\widetilde{k}}_j\stackrel{(2)}{=}C\ep^{\widetilde{i}}_l$ and the action of $Re_i$ and $Ce_j$ happens in the $2 \times 2$ submatrix with rows indexed by $j, \; j+1$ and columns $i,\; i+1$:
	\[M^{i,i+1}_{j,j+1}=
	\left[\begin{array}{cc}
	0 & 1 \\
	1 & b
	\end{array}\right]
	\]
	
	If $b=0$, then $C\ep^{\widetilde{i}}_j=C\ep^{\widetilde{i+2}}_j$ for $i+2\leq m$, or $C\ep^{\widetilde{i}}_j=0$, which contradicts that $Ce_j$ acts in the $k=i$--th column. So, we must have $b=1$. Then the actions commute in the following way, $M^{i,i+1}_{j,j+1}$ being in the upper left corner:
	\[
	\begin{tikzcd}[ampersand replacement=\&]
	\begin{array}{cc} 0 & 1 \\ 1 & 1 \end{array} \arrow{r}{Re_i} \arrow[swap]{d}{Ce_j} \& 
	\begin{array}{cc} 1 & 0 \\ 1 & 1 \end{array} \arrow{d}{Ce_j} \\
	\begin{array}{cc} 1 & 1 \\ 0 & 1 \end{array} \arrow{r}{Re_i} \& 
	\begin{array}{cc} 1 & 1 \\ 1 & 0 \end{array}  \\
	\end{tikzcd}
	\vspace{-0.75cm}
	\]
\end{proof}

\section{Skew Howe Duality}
\label{sec:skew}
\subsection{Skew Howe duality for crystals}
\label{sub:shc}
As we saw in Section \ref{sec:tensor}, the set of matrices:
$$\Lambda^NB_{n,m}=\{M \in \text{Mat}_{n \times m}(\{0,1\}) \; | \; \# 1's = N\}$$ 
is the $\mgl_n \times \mgl_m$--crystal for the representation $\Lambda^N(\C^n \otimes \C^m)$. There is an isomorphism of $\mgl_n \times \mgl_m$--crystals, whose construction is also discussed by van Leeuwen \cite{vL} (Section 3) and Danilov--Koshevoy \cite{DK} (Section 6), also known as \textbf{skew Howe duality} for crystals:
\begin{equation}\Lambda^NB_{n,m} \cong \bigsqcup_{\stackrel{\la \subset n \times m}{|\la|=N}}\mathcal{B}^{\mgl_n}_{\la} \otimes \mathcal{B}^{\mgl_m}_{\la^{\text{tr}}}\label{eq:skewhowe}\end{equation}  
This equivalence is analogous to the equivalence on the level of representations, and the categorification of skew Howe duality is studied by Ehrig and Stroppel in \cite{ES13}. We can express crystal skew Howe duality in Equation (\ref{eq:skewhowe}) either by mapping an element of $\Lambda^NB_{n,m}$ to a pair of $n\times m$ matrices or to a pair of tableaux of transpose shape. We do so by considering below only the first map or the composition of the first and second maps as follows, for $M \in \Lambda^NB_{n,m}$:
\begin{equation}\label{eqn:maps}M \xrightarrow{(Re^{\text{max}},Cf^{\text{max}})} (P,Q) \xrightarrow{(\phi,\psi)} (T_P,T_Q)\end{equation}

The operation $Re^{\text{max}}$ represents a sequence of raising operations $Re_i$ that maps $M$ to a $\mgl_m$--highest weight matrix $P$, and similarly $Cf^{\text{max}}$ represents a sequence of lowering operations $Cf_j$ that maps $M$ to a $\mgl_n$--lowest weight matrix $Q$. The final answer $(P,Q)$ does not depend on the particular sequence of operations we apply, as each connected component of a crystal has a single highest and lowest weight element.

\begin{remark}
	In Section 3 of \cite{vL}, the map $(Re^{\text{max}},Ce^{\text{max}})$ is considered instead.
\end{remark}

\begin{lemma}[See also \cite{vL}, Lemma 3.1.2]\label{lem:part}
	An element $M \in \Lambda^N B_{n,m}$ which is both $\mgl_n$-- and $\mgl_m$--highest weight must be a matrix with all $N$ 1's in the upper left corner filling the shape of a partition $\la_M$. Similarly, if $L\in \Lambda^N B_{n,m}$ is $\mgl_n$--lowest weight and $\mgl_m$--highest weight, then all $N$ 1's in the matrix for $L$ fill the shape of a partition $\la_L$ that is aligned Southwest.
\end{lemma}
\begin{proof}
	It suffices to show that pairs of the type $M_{i,j}=0, \; M_{i,j+1}=1$ and $M_{i,j}=0, \; M_{i+1,j}=1$ are not possible.
	
	Suppose first that a pair of the first type exists in $M$. Choose the one which is furthest Northeast, of the form $M_{i,j}=0, \; M_{i,j+1}=1$. If $i=1$, we can apply $Re_j$, which contradicts that $M$ is $\mgl_m$--highest weight. If we consider $i>1$, then by assumption $R\ep^i_j=1+\sum^{i-1}_{s=1}{M_{s,j+1}}-\sum^{i-1}_{s=1}{M_{s,j}} \leq 0$, so we must have $\sum^{i-1}_{s=1}{M_{s,j+1}} < \sum^{i-1}_{s=1}{M_{s,j}}$. For that to be true, we must have $M_{k,j+1}=0$ for at least one $1 \leq k \leq i-1$. Take the lowest (Southernmost) such entry, then $M_{k+1,j+1}=1$ and since our $(0,1)$ pair was the furthest Northeast, we must have $M_{k,s}=0 \; \forall \; s=j+2,\hdots,m$ (possibly $j+2>m$ and there are no such elements). But this tells us that $C\ep^{\widetilde{j+1}}_k\geq 1$, which contradicts the assumption that $M$ is $\mgl_n$--highest weight.
	
	Similarly, suppose a pair of the second type exists and pick the one which is the furthest Northeast, $M_{i,j}=0, \; M_{i+1,j}=1$. If $j=m$, then we can apply $Ce_i$, which contradicts the assumption that $M$ is $\mgl_n$--highest weight. If we consider $j<m$, then by our assumption, $C\ep^{\widetilde{j}}_i = 1 + \sum^{m}_{s=j+1}{M_{i+1,s}} - \sum^{m}_{s=j+1}{M_{i,s}} \leq 0$, hence $\sum^{m}_{s=j+1}{M_{i+1,s}}<\sum^{m}_{s=j+1}{M_{i,s}}$. This can only occur if $M_{i,k}=1$ for at least one $i+1 \leq k \leq m$. Pick the left-most such. Then $M_{i,k-1}=0$ and since $M_{i,j}=0, \; M_{i+1,j}=1$ was the furthest Northeast such pair, we must have $M_{s,k}=1 \; \; \forall \; 1 \leq s \leq i-1$. Hence, $R\ep^i_{k-1}\geq 1$, which contradicts the assumption that $M$ is $\mgl_m$--highest weight.
	
	The case when a matrix $L$ is $\mgl_n$--lowest weight and $\mgl_m$--highest weight is analogous, and there we can show that pairs of the type $L_{i,j}=0, \; L_{i,j+1}=1$ and $L_{i,j}=1, \; L_{i+1,j}=0$ are not possible.	
\end{proof}

The map $\phi: P \mapsto T_P$ in Equation (\ref{eqn:maps}) maps the $\mgl_m$--highest weight matrix $P$ to a top-left aligned semi-standard Young tableau (SSYT) $T_P$ with $N$ boxes in the following way (see \cite{vL}[Section 1.1]):

$T_P$ is defined by a sequence of partitions $\emptyset =\la^{(0)} \subset \la^{(1)} \subset \la^{(2)} \subset \hdots \subset \la^{(n)}=\la$. Each of the partitions is determined by $(\la^{(i)})^{\text{tr}}=\sum^i_{j=1}{rP_j} \in \mathbb{Z}^m_{\geq 0}$,  where $rP_j$ denotes the $j$--th row of the matrix $P$. The SSYT $T_P$ has shape $\la \subset n \times m$ and filling obtained by placing the number $i$ in each box of the skew-shape $\la^{(i)}/\la^{(i-1)}$, for each $i=1,2,\hdots,n$. 

Indeed, since $P$ is of size $n \times m$, by construction $\la^{\text{tr}} \subset m \times n$. Furthermore, we have that $(\la^{(i)})^{\text{tr}}=((\la^{(i)})^{\text{tr}}_1,(\la^{(i)})^{\text{tr}}_2,\hdots,(\la^{(i)})^{\text{tr}}_m)$ is a partition for each $i=1,\hdots,n$ because $P$ is $\mgl_m$--highest weight, i.e. $R\epsilon_k(P)=0$ for each $k=1,2,\hdots,m-1$, which tells us in particular:
\[0 \geq R\epsilon^i_k(P)=\sum^i_{j=1}{(P_{j,k+1}-P_{j,k})} + \delta_{P_{i,k},1}\delta_{P_{i,k+1},0}\] 
This in turn implies that $(\la^{(i)})^{\text{tr}}_k =\sum^{i}_{j=1}{P_{j,k}} \geq \sum^{i}_{j=1}{P_{j,k+1}}= (\la^{(i)})^{\text{tr}}_{k+1}$. Because the skew-shape $(\la^{(i)})^{\text{tr}}/(\la^{(i-1)})^{\text{tr}}$ has at most one box in each row (as the entries of $P$ are $0$ and $1$), $\la^{(i)}/\la^{(i-1)}$ has at most one box in each column, so each number $i$ appears at most once in each column. Together with $\la^{(i-1)}\subset \la^{(i)}$, we get that $T_P$ is strictly increasing in the columns and weakly increasing in the rows, so it is a SSYT.

The map $\psi: Q \mapsto T_Q$ in Equation (\ref{eqn:maps}) is defined similarly. For a $\mgl_n$--lowest weight matrix $Q=Cf^{\text{max}}(M)$, we get a SSYT $T_Q$ as follows. The shape $\la \subset m\times n$ and its filling are determined analogously to $\phi$, with the difference that $(\la^{(i)})^{\text{tr}}=\sum^i_{j=1}\overline{cQ}_j \in \mathbb{Z}^n_{\geq 0}, \; \forall \; i=1,2,\hdots,m$, where $\overline{cQ}_j$ denotes the $j$--th column of $Q$ traversed from bottom to top ($Q_{i,j}=\overline{Q}_{n+1-i,j}$). Namely, $(\la^{(i)})^{\text{tr}}=((\la^{(i)})^{\text{tr}}_1,(\la^{(i)})^{\text{tr}}_2,\hdots,(\la^{(i)})^{\text{tr}}_n)$ where $(\la^{(i)})^{\text{tr}}_k=\sum^i_{j=1}Q_{n+1-k,j}$. The matrix $Q$ is $\mgl_n$--lowest weight, i.e. $C\phi_{n-k}(Q)=0 \; \forall \; k=1,2,\hdots,n-1$, so in particular (for $\widetilde{i}=m+1-i$): 
\[0 \geq C\phi^{\widetilde{i}}_{n-k}(Q)=\sum^{i}_{j=1}{(Q_{n-k,j}-Q_{n+1-k,j})} + \delta_{Q_{n+1-k,i},1}\delta_{Q_{n-k,i},0} \] 
This in turn implies $(\la^{(i)})^{\text{tr}}_k =\sum^{i}_{j=1}{Q_{n+1-k,j}} \geq \sum^{i}_{j=1}{Q_{n-k,j}}= (\la^{(i)})^{\text{tr}}_{k+1}$. So $(\la^{(i)})^{\text{tr}}$ and hence $\la^{(i)}$ is a partition for all $i$. The remaining arguments that $T_Q$ is a SSYT are identical to those for $T_P$.

Moreover, using the two maps $\phi$ and $\psi$, the tableaux we obtain from the pair of matrices $(P,Q)=(Re^{\text{max}}(M),Cf^{\text{max}}(M))$ are of transpose shape.

\begin{lemma} The partitions $(\la_P,\la_Q)$ associated to the two matrices $(P,Q)$ as above are of transpose shapes: $(\la_P)^{\text{tr}}=\la_Q$.	
\end{lemma}

\begin{proof}
	Let $L$ be the element of $\Lambda^NB_{n,m}$ given by: $$L= Cf^{\text{max}}(P)=Cf^{\text{max}}Re^{\text{max}}(M)=Re^{\text{max}}Cf^{\text{max}}(M)=Re^{\text{max}}(Q).$$ Then, by Lemma \ref{lem:part}, the $N$ $1$'s of $L$ fill the shape of a partition $\la_L$ aligned Southwest and so:
	\begin{align*}
	\la_L &=\DS{\left(\sum^m_{k=1}{L_{n,k}},\hdots,\sum^m_{k=1}{L_{2,k}}, \sum^m_{k=1}{L_{1,k}}\right)} \\
	(\la_L)^{\text{tr}} &=\DS{\left(\sum^n_{j=1}{L_{j,1}},\hdots,\sum^n_{j=1}{L_{j,m-1}}, \sum^n_{j=1}{L_{j,m}}\right)}
	\end{align*}
	
	\setlength\columnsep{5pt}
	\begin{multicols}{2}
		\begin{tikzpicture}[baseline]
		\matrix[matrix of math nodes,left delimiter={[}, right delimiter={]}](n)
		{
			0 & 0 & 0 & \hdots & 0 & 0 & 0 & 0 \\
			1 & 1 & 0 & \hdots & 0 & 0 & 0 & 0  \\
			\vdots & \phantom{0} & \phantom{0} & \ddots  & \phantom{0} & \phantom{0} & \phantom{0} & \vdots  \\
			1 & 1 & 1 & \hdots  & 1 & 0 & 0 & 0 \\
			1 & 1 & 1 & \hdots  & 1 & 1 & 1 & 0 \\
		};
		\node[left] at ($(n.west)+(-0.2,0)$) {$L=$};
		\end{tikzpicture}~\begin{tikzpicture}[baseline]
		\matrix[matrix of math nodes,left delimiter={[}, right delimiter={]}](m)
		{
			\phantom{0} &  \phantom{0} & \phantom{0} & \phantom{0}  & \phantom{0} & \phantom{0} & \phantom{0} & \phantom{0} \\
			\phantom{0} & \phantom{0} & \phantom{0}  & \phantom{0} & \phantom{0} & \phantom{0} & \phantom{0} & \phantom{0}  \\
			\phantom{0} & \phantom{0} & \vdots  & \phantom{0} & \phantom{0} & \phantom{0} & \phantom{0} & \phantom{0}  \\
			\phantom{0} & \phantom{0} & \la_{L,2} & \phantom{0} & \phantom{0} & \phantom{0} & \phantom{0} & \phantom{0} \\
			\phantom{0} & \phantom{0} & \la_{L,1} & \phantom{0} & \phantom{0} & \phantom{0} & \phantom{0} & \phantom{0} \\
		};
		\draw[color=red](m-5-1.south west) -- (m-1-1.south west);
		\draw[color=red](m-5-1.south west) -- (m-5-8.south west);
		\draw[color=red] (m-5-8.south west) -- (m-4-8.south west);
		\draw[color=red] (m-4-8.south west) -- (m-4-6.south west);
		\draw[color=red] (m-4-6.south west) -- (m-3-6.south west);
		\draw[color=red] (m-3-6.south west) -- (m-3-5.south west);
		\draw[color=red] (m-3-5.south west) -- (m-2-5.south west);
		\draw[color=red] (m-2-5.south west) -- (m-2-3.south west);
		\draw[color=red] (m-2-3.south west) -- (m-1-3.south west);
		\draw[color=red] (m-1-3.south west) -- (m-1-1.south west);
		\node[left] at ($(m.west)+(-0.2,0)$) {$\la_L=$};
		\end{tikzpicture}
	\end{multicols}\vspace{-0.2cm}
	By definition, we have:
	\begin{align*}
		(\la_P)^{\text{tr}}&=\sum^n_{j=1}rP_j=\left(\sum^n_{j=1}{P_{j,1}},\hdots,\sum^n_{j=1}{P_{j,m}}\right)\\
		&=R\wt(P)\stackrel{(\ast)}{=}R\wt(L)=\left(\sum^n_{j=1}{L_{j,1}},\hdots,\sum^n_{j=1}{L_{j,m}}\right) =(\la_L)^{\text{tr}}\end{align*}
	where $(\ast)$ follows from the fact that $Cf_i$ is a $\mgl_m$--crystal morphism for all $i$. Similarly,
	\begin{align*}
		(\la_Q)^{\text{tr}}&=\sum^m_{k=1}c\overline{Q}_k=\left(\sum^m_{k=1}{Q_{n,k}},\hdots,\sum^m_{k=1}{Q_{1,k}}\right)\\
		&=\overline{C\wt(Q)}\stackrel{(\ast\ast)}{=}\overline{C\wt(L)}=\left(\sum^m_{k=1}{L_{n,k}},\hdots,\sum^m_{k=1}{L_{1,k}}\right) =\la_L\end{align*}
	where $\overline{C\wt(A)}=\overline{(a_1,\hdots,a_n)}=(a_n,\hdots,a_1)$ for a matrix $A$, and $(\ast\ast)$ follows from the fact that $Re_j$ is a $\mgl_n$--crystal morphism for all $j$. Therefore, indeed $(\la_P)^{\text{tr}}=\la_Q$.
\end{proof}

\begin{example}\label{ex:tabl} For the following matrix in  $\Lambda^9 B_{3,5}$, we apply $(Re^{\text{max}},Cf^{\text{max}})$ to get the corresponding pair of matrices, and $(\phi,\psi)$ to get a pair of tableaux of transpose shapes $\la$ and $(\la)^{\text{tr}}$, where $\la=(5,3,1)$.
	\vspace{0.5cm}
	
	\begin{align*} \left[\begin{array}{ccccc}
	1 & 1 & 1 & 0 & 0  \\
	0 & 0 & 1 & 1 & 0  \\
	1 & 1 & 1 & 0 & 1  \\
	\end{array}\right] &\xrightarrow{(Re^{\text{max}},Cf^{\text{max}})} 
	\left[ \begin{array}{ccccc}
	1 & 1 & 1 & 0 & 0  \\
	1 & 0 & 0 & 1 & 0  \\
	1 & 1 & 1 & 0 & 1  \\
	\end{array}\right] \otimes \left[ \begin{array}{ccccc}
	0 & 0 & 1 & 0 & 0  \\
	1 & 1 & 1 & 0 & 0  \\
	1 & 1 & 1 & 1 & 1  \\
	\end{array}\right] \\
	& \\
	&  \xrightarrow{(\phi,\psi)} \begin{minipage}[t]{0.02\textwidth}\vspace{-0.6cm}\ytableausetup{nosmalltableaux}
	\begin{ytableau}
	1 & 1 & 1 & 2 & 3 \\
	2 & 3 & 3 \\
	3
	\end{ytableau}\;\;$\otimes$~\ytableausetup{nosmalltableaux}
	\begin{ytableau}
	1 & 1 & 3 \\
	2 & 2 \\
	3 & 3 \\
	4 \\
	5
	\end{ytableau}\end{minipage}
	\end{align*}
\end{example}

Next, we verify that the map in Equation (\ref{eqn:maps}) is a $\mgl_n\times \mgl_m$--crystal morphism. We already know that $Re^{\text{max}}$ is a $\mgl_n$--crystal morphism and $Cf^{\text{max}}$ is a $\mgl_m$--crystal morphism, so it remains to check that for $\phi$ and $\psi$.

\begin{lemma}\label{lem:ficom}
	The map defined above \[\phi:\{P \in \Lambda^N B_{n,m}\; | \; Re_i(P)=0 \quad \forall \; i=1,2,\hdots,m-1\} \longrightarrow \bigsqcup_{\stackrel{\la \in n\times m}{ |\la|=N}}\mathcal{B}^{\mgl_n}_{\la}\] is a $\mgl_n$--crystal morphism.
\end{lemma}

\begin{remark}
	\label{rem:crtableau}
	The crystal $\mathcal{B}^{\mgl_n}_{\la}$ with $\la=(\la_1,\la_2,\hdots,\la_n) \in \Z^n_{\geq 0}$, where $\la_1 \geq \la_2 \geq \hdots \geq \la_n$, corresponds to the irreducible $\mgl_n$--representation with highest weight $\la=\la_1\epsilon_1+\la_2\epsilon_2+\hdots+\la_n\epsilon_n$ (see \cite{HoKa}). The elements of the crystal are the SSYT with shape $\la$ and filling from the alphabet $\{1,2,\hdots,n\}$. The Kashiwara operators $Ce_i,Cf_i$ act by: For each tableau, create a list by looking at the columns from left to right and adding a ``$+$'' if the column contains $i+1$, and a ``$-$'' if the column contains $i$ (if it contains both, put ``$+$'' in front of ``$-$''). Then consecutively cancel all the ``$+-$'' pairs until all the ``$-$'' are on the left and all the ``$+$'' are on the right. Then $C\ep_i$ is equal to the number of unpaired ``$+$'' and $C\phi_i$ is the number of unpaired ``$-$''. The operator $Ce_i$ then changes the $i+1$ corresponding to the leftmost unpaired ``$+$'' to an $i$, and $Cf_i$ changes the $i$ corresponding to the rightmost ``$-$'' to an $i+1$. The weight of a tableau can be recovered by considering how many times each number appears in the filling, it is $(\# 1's, \# 2's, \hdots, \# n's)$. \cite{HoKa}
\end{remark}

\begin{example}\label{ex:crmor}
	Continuing the setting in Example \ref{ex:tabl}, we show that the map $\phi$ respects the action of $Ce_2$.
	
	\[
	\begin{tikzcd}[ampersand replacement=\&]
	{\small \left[ \begin{array}{ccccc}
		1 & 1 & 1 & 0 & 0  \\
		1 & 0 & 0 & 1 & 0  \\
		1 & 1 & 1 & 0 & 1  \\
		\end{array}\right]}  \arrow{r}{Ce_2} \arrow[swap]{d}{\phi} \& 
	{\small \left[ \begin{array}{ccccc}
		1 & 1 & 1 & 0 & 0  \\
		1 & 1 & 0 & 1 & 0  \\
		1 & 0 & 1 & 0 & 1  \\
		\end{array}\right]} \arrow{d}{\phi} \\
	\begin{ytableau}
	1 & 1 & 1 & 2 & 3 \\
	2 & 3 & 3 \\
	3
	\end{ytableau}  \arrow{r}{Ce_2} \& 
	\begin{ytableau}
	1 & 1 & 1 & 2 & 3 \\
	2 & 2 & 3 \\
	3
	\end{ytableau}
	\end{tikzcd}
	\]
	
\end{example}

\begin{proof}{(of Lemma \ref{lem:ficom})}
	Let:
	$$\Lambda^N_R B_{n,m}=\{P \in \Lambda^N B_{n,m}\; | \; Re_i(P)=0 \quad \forall \; i=1,2,\hdots,m-1\}$$ and for each $\la \subset n \times m, \; |\la|=N$, let $(\Lambda^N_R B_{n,m})_{\la}$ denote the subset of elements with highest (row) weight $\la^{\text{tr}}$, i.e.
	\[(\Lambda^N_R B_{n,m})_{\la}=\left\{P \in \Lambda^N_R B_{n,m}\;  : \; \left(\sum^n_{i=1}P_{i,1},\hdots,\sum^n_{i=1}P_{i,m}\right)=(\la)^{\text{tr}}\right\}\]
	Then we get two injective $\mgl_n$--crystal morphisms shown below. The first one is the restriction of the map in Equation (\ref{eqn:col}) on page 9, while the second maps any tableau in $\mathcal{B}^{\mgl_n}_{\la}$ to a tensor product of its columns, taken from right to left, as this gives a so-called {\sl admissible reading} of the tableau (see \cite{HoKa}, Theorem 7.3.6).

		\begin{align*}
		&(\Lambda^N_R B_{n,m})_{\la}  \hookrightarrow \mathcal{B}^{\mgl_n}_{\varpi_{(\la^{\text{tr}})_m}} \otimes \hdots \otimes \mathcal{B}^{\mgl_n}_{\varpi_{(\la^{\text{tr}})_1}} \\
		&\left[\begin{array}{ccc}
		| & & | \\
		c_1 & \hdots & c_m \\
		| & & | 
		\end{array}\right] \mapsto c_m \otimes \hdots \otimes c_1 \\
		& \\
		&\mathcal{B}^{\mgl_n}_{\la}  \hookrightarrow \mathcal{B}^{\mgl_n}_{\varpi_{(\la^{\text{tr}})_m}} \otimes \hdots \otimes \mathcal{B}^{\mgl_n}_{\varpi_{(\la^{\text{tr}})_1}} \\
		&T \mapsto T_m \otimes \hdots \otimes T_1 \quad \text{($T_i=$ $i$--th column of the tableau $T$)}
		\end{align*}
	
	The map $\phi$ makes the following diagram commute, and so is itself a $\mgl_n$--crystal morphism:
	\begin{center}
		\begin{tikzcd}[column sep=small]
			(\Lambda^N_R B_{n,m})_{\la} \arrow{rr}{\phi} \arrow[swap]{dr}{}& &\mathcal{B}^{\mgl_n}_{\la} \arrow{dl}{}\\
			& \mathcal{B}^{\mgl_n}_{\varpi_{(\la^{\text{tr}})_m}} \otimes \hdots \otimes \mathcal{B}^{\mgl_n}_{\varpi_{(\la^{\text{tr}})_1}} & 
		\end{tikzcd}
	\end{center}  
\end{proof}

\begin{lemma}
	The map defined above $$\psi:\{Q \in \Lambda^NB_{n,m}\; | \; Cf_i(Q)=0 \quad \forall \; i=1,2,\hdots,n\} \longrightarrow \bigsqcup_{\stackrel{\la \in n\times m}{ |\la|=N}}\mathcal{B}^{\mgl_m}_{\la^{\text{tr}}}$$ is a $\mgl_m$--crystal morphism.
\end{lemma}

\begin{proof}
	We want to show that for any $Re_k, \; k=1,2,\hdots, m-1$, and any $Q \in \Lambda^N B_{n,m}$ such that $Cf_i(Q)=0 \; \forall \; i=1,2,\hdots,n-1$ and $Re_k(Q) \neq 0$, the following diagram commutes. The proof is analogous for $Rf_k$.
	\[
	\begin{tikzcd}[ampersand replacement=\&]
	Q \arrow{r}{Re_k} \arrow[swap]{d}{\psi} \& 
	Q' \arrow{d}{\psi} \\
	T_Q  \arrow{r}{Re_k} \& 
	T_{Q'}
	\end{tikzcd}
	\]
	When considering the action of $Re_k$, we are looking at columns $k$ and $k+1$ of $Q$. The setup here is the same as in Lemma \ref{lem:ficom} if we consider $Q$ rotated $90^{\circ}$ clockwise, and everything else follows through analogously, including that $\psi$ preserves the values of $R\ep_k, \; R\phi_k$, and $R\wt$.
\end{proof}

\vspace{0.5cm}
Finally, we will check that the map is indeed an isomorphism, i.e. it remains to show it is a bijection on sets.

\begin{lemma}[See also \cite{vL}, Section 3]
	The $\mgl_n \times \mgl_m$--crystal morphism $\Lambda^NB_{n,m} \longrightarrow \bigsqcup_{\stackrel{\la \subset n \times m}{|\la|=N}}\mathcal{B}^{\mgl_n}_{\la} \otimes \mathcal{B}^{\mgl_m}_{\la^{\text{tr}}}$ given by the maps below is an isomorphism:
	\[M \xrightarrow{(Re^{\text{max}},Cf^{\text{max}})} P \otimes Q \xrightarrow{(\phi,\psi)} T_P \otimes T_Q\]
\end{lemma}
\begin{proof}
	Let us first consider the map $(Re^{\text{max}},Cf^{\text{max}})$. As it doesn't map any nonzero element to zero, it suffices to check it is surjective. For a given $\la \subset n \times m,\;|\la|=N$, take any pair $P \otimes Q \in \mathcal{B}^{\mgl_n}_{\la} \otimes \mathcal{B}^{\mgl_m}_{\la^{\text{tr}}}$. Namely, $P, Q \in \text{Mat}_{n \times m}(\{0,1\})$ such that $Re_i(P)=0 \; \forall \; i=1,2,\hdots,m-1$, $Cf_j(Q)=0 \; \forall \; j=1,2,\hdots, n-1$ and $\la^{\text{tr}}=\sum^n_{s=1}rP_s$, $\la=\sum^m_{s=1}\overline{cQ}_s$. Then, as discussed in Lemma \ref{lem:part}, $Cf^{\text{max}}P=Re^{\text{max}}Q=L$ is the matrix with $N$ $1$'s filling the shape of $\la$ aligned Southwest. Suppose in particular that two sequences that realize $L$ are:
	\[Cf_{j_1}Cf_{j_2} \hdots Cf_{j_p}P=Re_{i_1} Re_{i_2} \hdots Re_{i_q}Q\]
	Using the commutativity properties of the row and column Kashiwara operators, we take:
	\[M:=Rf_{i_q} \hdots Rf_{i_2} Rf_{i_1}P = Ce_{j_p} \hdots Ce_{j_2} Ce_{j_1}Q \neq 0\]
	Then $M \in \Lambda^N{B_{n,m}}$  has the property that $(Re^{\text{max}},Cf^{\text{max}})(M)=P \otimes Q$. When considering $(\phi,\psi)$, given a tableau $T_P \in \mathcal{B}^{\mgl_n}_{\la}$, we can recover the $n \times m$ matrix $P$ such that $\phi(P)=T_P$ in the following way: For each $k=1,\hdots,n$, look at the entries labelled $k$ in $T_P$. If such an entry appears in the $j$--th column of the tableau, then set $P_{k,j}=1$ while the remaining elements in the $k$--th row of the matrix are set to zero.
	Similarly, given a tableau $T_Q \in \mathcal{B}^{\mgl_m}_{\la^{\text{tr}}}$, we can recover the $n\times m$ matrix $Q$ such that $\psi(Q)=T_Q$ as follows: For each $l=1,\hdots,m$, look at the entries labelled $l$ in $T_Q$. If such an entry appears in the $j$--th column of the tableau, then set $Q_{n+1-j,l}=1$ while the remaining elements in the $l$--th column of the matrix $Q$ are set to zero.
\end{proof}
\subsection{Two cactus group actions}

Henriques and Kamnitzer \cite{HK} define an action of the cactus group $C_n=C_{\mgl_n}$ on $n$--tensor products of $\fg$--crystals (for any complex reductive Lie algebra $\fg$). For any $s_{p,q} \in C_n$, $1 \leq p < q \leq n$ corresponding to the Dynkin subdiagram $\{p,\hdots,q-1\}$ of $\mgl_n$,  and any tensor product $B_1 \otimes \hdots \otimes B_n$, where each $B_i$ is a $\fg$--crystal, we get a map:
\[s^{\text{o}}_{p,q}: B_1\otimes \hdots \otimes B_p \otimes \hdots \otimes B_q \otimes B_n \rightarrow B_1\otimes \hdots \otimes B_q \otimes \hdots \otimes B_p\otimes B_n\]
For any element $b_1 \otimes \hdots \otimes b_n$ of $B_1 \otimes B_2 \otimes \hdots \otimes B_n$:
\[\begin{array}{c}s^{\text{o}}_{p,q}(b_1 \otimes b_2 \otimes \hdots \otimes b_n)= \\
\hspace{-0.2cm}=b_1 \otimes \hdots b_{p-1} \otimes\xi(\xi(b_q) \otimes \xi(b_{q-1}) \otimes \hdots \otimes \xi(b_{p+1}) \otimes \xi(b_p)) \otimes b_{q+1} \otimes \hdots \otimes b_n\end{array}\]
Here $\xi$ is the Sch\"{u}tzenberger involution as in Definition \ref{def:schu}. If $\text{flip}^{p,q}$ denotes the permutation that flips the interval $\left[p,q\right]$, then $s_{p,q}$ acts as:
\[s^{\text{o}}_{p,q}=\text{Id} \otimes \hdots \otimes \text{Id} \otimes \xi_{p,q} \circ (\xi \otimes \hdots \otimes \xi) \circ \text{flip}^{p,q} \otimes \text{Id} \otimes \hdots \otimes \text{Id}\]
We claim that this ``outer'' cactus group action agrees with the ``inner'' cactus group action from Section \ref{sec:cactus} when considered on the crystal $\Lambda^NB_{n,m}$, using skew Howe duality for $\mgl_n$ and $\mgl_m$.

\begin{lemma}
	\label{lem:flip}
	For the Sch\"{u}tzenberger involution $\xi=\xi_{\mathcal{B}^{\mgl_m}_{\varpi_k}}$, and any element $b=(b_1, b_2, \hdots, b_m) \in \mathcal{B}^{\mgl_m}_{\varpi_k}$, $\varpi_k$ being the $k$--th fundamental weight of $\mgl_m$:  \[\xi(b)_i=b_{m+1-i}, \quad \quad \text{ i.e. } \quad \xi(b_1,\; b_2, \; \hdots, \; b_m)=(b_m,\; \hdots,\; b_2, \; b_1)\]
\end{lemma}
\begin{proof}
	An element in $\mathcal{B}^{\mgl_m}_{\varpi_k}$ can be expressed as a vector in $(\{0,1\})^m$ with $k$ of its entries equal to $1$. If they occupy positions $1\leq i_1 < i_2 < \hdots < i_k \leq m$, that element can be represented as $\{i_1,\hdots, i_k\}$. In particular, the highest and lowest weight elements are $(1, 1, \hdots, 1, 0, \hdots, 0)$ and $(0,\hdots, 0, 1, \hdots, 1)$, represented by $\{1, 2,\hdots, k\}$ and $\{m+1-k,\hdots, m-1, m\}$ respectively. Additionally, let $f_{j_1,\hdots,j_s}:=f_{j_s} \hdots f_{j_1}$ and $e_{j_1,\hdots,j_s}:=e_{j_s} \hdots e_{j_1}$ represent a composition of Kashiwara operators for the crystal $\mathcal{B}^{\mgl_m}_{\varpi_k}$. Then any element $\{i_1,\hdots, i_k\}$ can be expressed in terms of the highest weight element and a sequence of lowering Kashiwara operators, so we get:
	\begin{align*}
		&\{i_1,\hdots, i_k\}=f_{1,2,\hdots,i_1-1}\hdots f_{k-1,k,\hdots,i_{k-1}-1}f_{k,k+1,\hdots,i_k-1}\{1, 2, \hdots, k\}  \\
		&\xrightarrow{\xi} e_{m-1,m-2,\hdots,m-i_1+1}\hdots e_{m-k,m-k-1,\hdots,m-i_k+1}\{m-k+1, \hdots, m\} \\
		&=\{m+1-i_k, m+1-i_{k-1},\hdots, m+1-i_1\}
	\end{align*}
\end{proof}

\begin{lemma}\label{lem:xinslm} Consider the $\mgl_k$--crystal Sch\"{u}tzenberger involution denoted $\xi=\xi^{\mgl_m}_{1,k}: \Lambda^N{B_{n,m}} \longrightarrow \Lambda^N{B_{n,m}}$, corresponding to the Dynkin subdiagram $\overline{k}=\{1,\hdots,k-1\}$ of $\mgl_m$. Then $\xi$ is a $\mgl_n$--crystal morphism. Analogously, the $\mgl_l$--crystal  Sch\"{u}tzenberger involution $\xi^{\mgl_n}_{1,l}$ is a $\mgl_m$--crystal morphism.
\end{lemma}
\begin{proof}
	We will focus on the commutativity of $\xi$ with $Ce_i$, $i \in \{1,\hdots,n-1\}$, as the remaining morphism properties are analogous. Take any nonzero element $ b \in \Lambda^N{B_{n,m}}$ such that $Ce_i(b)\neq 0$. Then $b$ belongs to some irreducible $\mgl_k$--crystal $\mathcal{B}^{\mgl_k}_\la \subset \left(\Lambda^N{B_{n,m}}\right)\left|_{\overline{k}}\right.$, i.e. a connected component of the restriction of $\Lambda^N{B_{n,m}}$ to $\overline{k}$, with highest weight $\la$, say, and highest and lowest weight elements $b_{\la}$ and $b^{\text{low}}_{\la}$, respectively (in terms of the row crystal operations).  We want to check that $\xi Ce_i(b)=Ce_i\xi(b)$. We expand each side using that there is a sequence of Kashiwara operators such that $b=Re_{j_s}Re_{j_{s-1}}\hdots Re_{j_1}(b^{\text{low}}_{\la})$, where $j_1,\hdots,j_s \in \overline{k}$, together with $Ce_iRe_j=Re_jCe_i, \; Ce_iRf_j=Rf_jCe_i$ from Proposition \ref{prop:mor} and $\xi Re_j=Rf_{k-j}\xi, \; \xi(b^{\text{low}}_{\la})=b_{\la}, \; \xi(b_{\la})=b^{\text{low}}_{\la}$ from Definition \ref{def:schu}. In this way, we get:
	\[b\xrightarrow{Ce_i}
	Re_{j_s}Re_{j_{s-1}}\hdots Re_{j_1}Ce_i(b^{\text{low}}_{\la}) \xrightarrow {\xi}Rf_{k-j_s}Rf_{k-j_{s-1}}\hdots Rf_{k-j_1}\xi Ce_i(b^{\text{low}}_{\la})\]
	\[b\xrightarrow{\xi}Rf_{k-j_s}Rf_{k-j_{s-1}}\hdots Rf_{k-j_1}(b_{\la})\xrightarrow{Ce_i}Rf_{k-j_s}Rf_{k-j_{s-1}}\hdots Rf_{k-j_1}Ce_i(b_{\la})\]
	It remains to show that $\xi Ce_i(b^{\text{low}}_{\la})=Ce_i(b_{\la})$. Since $Ce_i(b) \neq 0$, we also have $Ce_i(b_{\la}) \neq 0, \; Ce_i(b^{\text{low}}_{\la}) \neq 0$, as $Re_j$ and $Rf_j$ preserve $C\ep_i$, by Proposition \ref{prop:mor}. The proposition also tells us that $Ce_i$ commutes with $Re_j$ and preserves $R\ep_j, \; R\phi_j, \; R\wt$. In particular, this means $c_{\la}=Ce_i(b_{\la})$ and $c^{\text{low}}_{\la}=Ce_i(b^{\text{low}}_{\la})$ are the highest  and lowest weight elements of an irreducible $\mgl_k$--crystal in $\left(\Lambda^N{B_{n,m}}\right)_{\overline{k}}$, and so get exchanged by $\xi$.
\end{proof}

\begin{lemma}\label{lem:parcom} For all $k=1,\hdots,n$, $i=1, \hdots, k-1$, and $M \in \Lambda^N B_{n,m}$ s.t. $Ce_i(M) \neq 0$, respectively $Cf_i(M) \neq 0$, we have for $s_{1,k} \in C_n$:
	$$\so(Ce_i(M))=Cf_{k-i}(\so(M)) \quad \quad \so(Cf_i(M))=Ce_{k-i}(\so(M))$$
\end{lemma}
\begin{proof}
	Let $M_{\left[k\right]}$ denote the submatrix made up of the first $k$ rows of $M$, denoted $r_1, \hdots, r_k \in (\{0,1\})^m$. Since $\so$ and $Ce_i, \; Cf_i, \; i=1,\; \hdots, \; k-1$, are determined by and affect only the first $k$ rows of $M$, we have $\so(M)_{\left[k\right]}=\so(M_{\left[k\right]})$ and similarly for $Ce_i, \; Cf_i$ so we will work with $L:=M_{\left[k\right]}$.
	
	Let $p_k:=\xi^{\otimes k} \circ \text{flip}^{1,k}$ and note that $\so(Ce_i(L))=\xi^{\mgl_m}_{1,k} (p_k(Ce_i(L)))$. On the other hand $Cf_{k-i}(\so(L))=Cf_{k-i}(\xi^{\mgl_m}_{1,k} (p_k(L)))=\xi^{\mgl_m}_{1,k}(Cf_{k-i}(p_k(L)))$ since, as shown in Lemma \ref{lem:xinslm}, $\xi^{\mgl_m}_{1,k}$ is a $\mgl_n$--crystal morphism. So, we need to check that $p_k(Ce_i(L))=Cf_{k-i}(p_k(L))$. To verify this, we will first check that $C\ep^{\widetilde{j}}_i(L)=C\phi^{\widetilde{m+1-j}}_{k-i}(p_k(L))$. From the definition and Lemma \ref{lem:flip}, $p_k(L)_{s,t}=L_{k+1-s,m+1-t}$, therefore:
	\begin{align*}C\ep^{\widetilde{j}}_i(L)&=\sum^{m}_{r=j}{(L_{i+1,r}-L_{i,r})} + \delta_{L_{i,j},1}\delta_{L_{i+1,j},0} \\
		C\phi^{\widetilde{m+1-j}}_{k-i}(p_k(L)) &=\sum^{m+1-j}_{r=1}{(p_k(L)_{k-i,r}-p_k(L)_{k+1-i,r})} + \\
		& \quad \; + \delta_{p_k(L)_{k+1-i,m+1-j},1}\delta_{p_k(L)_{k-i,m+1-j},0} \\
		&=\sum^{m+1-j}_{r=1}{(L_{i+1,m+1-r}-L_{i,m+1-r})} + \delta_{L_{i,j},1}\delta_{L_{i+1,j},0} \\
		&=C\ep^{\widetilde{j}}_i(L)
	\end{align*}
	So, if $l$ is the closest number to $m$ where the maximum  $C\ep^{\widetilde{l}}_i(L)=\max_{\widetilde{j}}{C\ep^{\widetilde{j}}_i(L)}$ is achieved, then $m+1-l$ is the closest number to $1$ where $C\phi^{\widetilde{m+1-l}}_{k-i}(p_k(L))=\max_{\widetilde{j}}{C\phi^{\widetilde{m+1-j}}_{k-i}(p_k(L))}$ is achieved. Namely, if $Ce_i$ acts in column $l$ of $L$, then $Cf_{k-i}$ acts in column $m+1-l$ of $p_k(L)$, and:
	
	\begin{align*} 
		p_k(Ce_i(L))_{s,t} &=\begin{cases} 
			L_{k+1-s,m+1-t} & (s,t) \neq (k+1-i,m+1-l), \\
			& \hspace{1.3cm} (k-i,m+1-l) \\
			1 & (s,t)=(k+1-i,m+1-l) \\
			0 & (s,t)=(k-i,m+1-l) \\ 
		\end{cases} \\
		Cf_{k-i}(p_k(L))_{s,t} &=
		\begin{cases} 
			p_k(L)_{s,t} & (s,t) \neq (k+1-i,m+1-l), \\
			& \hspace{1.3cm} (k-i,m+1-l) \\
			1 & (s,t)=(k+1-i,m+1-l) \\
			0 & (s,t)=(k-i,m+1-l) \\ 
		\end{cases}\end{align*}
	\begin{align*}
		& =\begin{cases} 
			L_{k+1-s,m+1-t} & (s,t) \neq (k+1-i,m+1-l), \\
			& \hspace{1.3cm} (k-i,m+1-l) \\
			1 & (s,t)=(k+1-i,m+1-l) \\
			0 & (s,t)=(k-i,m+1-l) \\ 
		\end{cases}
	\end{align*}
\end{proof}

\begin{theorem}
	\label{thm:agree} When considered on $\Lambda^NB_{n,m}$, the outer action of $C_n$ on the $n$--tensor product of $\mgl_m$--crystals defined by Henriques-Kamnitzer agrees with the inner cactus group action on a $\mgl_n$--crystal under the crystal isomorphism: $\Lambda^NB_{n,m} \cong \bigsqcup_{\stackrel{\la \subset n \times m}{|\la|=N}}\mathcal{B}^{\mgl_n}_{\la} \otimes \mathcal{B}^{\mgl_m}_{\la^{\text{tr}}}$. More precisely, for any $s_{p,q} \in C_n$ and any $M \in \Lambda^N{B_{n,m}}$:
	\[s^{\text{o}}_{p,q}(M)=s^{\text{i}}_{p,q}(M)\]
\end{theorem}

\begin{proof}
	It suffices to check the claim on elements which are of the form $s_{1,k} \in C_n, \; k=2,\hdots,n$ since they generate the group. Indeed, from the cactus relations, we have that for any $1 \leq p < q \leq n$, $s_{p,q}=s_{1,q}s_{1,q+1-p}s_{1,q}$.
	
	We will use the matrix version of the isomorphism:
	\begin{equation}\label{eqn:rsk*}\Lambda^NB_{n,m} \cong \bigsqcup_{\stackrel{\la \subset n \times m}{|\la|=N}}\mathcal{B}^{\mgl_n}_{\la} \otimes \mathcal{B}^{\mgl_m}_{\la^{\text{tr}}}\end{equation}
	namely the map $(Re^{\max},Cf^{\max})$. On the left side of Equation (\ref{eqn:rsk*}), we have the tensor product action of $\so$ on $n$--tensors of $\mgl_m$--crystals:
	$$\Lambda^N B_{n,m} \cong \bigsqcup_{\underline{l} \in \mathbb{N}^n, |\underline{l}|=N} \mathcal{B}^{\mathfrak{gl}_m}_{\varpi_{l_1}} \otimes \hdots \otimes \mathcal{B}^{\mathfrak{gl}_m}_{\varpi_{l_n}}$$
	
	On the right side of Equation (\ref{eqn:rsk*}), we have the internal action of $\sii$ on any element $b \otimes c \in \mathcal{B}^{\mgl_n}_{\la} \otimes \mathcal{B}^{\mgl_m}_{\la^{\text{tr}}}$, given by:
	$$\sii\cdot (b \otimes c)= (\xi_{1,k} \otimes \text{id}) \; (b \otimes c)=\xi_{1,k}(b) \otimes c$$
	
	Since $s_{1,k} \in C_n$ acting (as $s^{\text{o}}_{1,k}$) on $n$--tensor products of $\mgl_m$--crystals is a $\mgl_m$--crystal isomorphism (see \cite{HK}, Theorem 2), it suffices to check the statement of the theorem on $\mgl_m$--highest weight elements, on which $Re^{\text{max}}$ acts as the identity. So we need to check that the following diagram commutes for every $P \in \Lambda^N B_{n,m} \; \text{s.t. } Re_i(P)=0$ for all $i=1,2,\hdots,m-1$:
	\[
	\begin{tikzcd}[ampersand replacement=\&]
	\Lambda^N B_{n,m} \arrow{rrr}{(Re^{\max},Cf^{\max})} \arrow[swap]{d}{\so} \& \& \&
	\bigsqcup_{\stackrel{\la \subset n \times m}{|\la|=N}}\mathcal{B}^{\mgl_n}_{\la} \otimes \mathcal{B}^{\mgl_m}_{\la^{\text{tr}}} \arrow{d}{\sii} \\
	\Lambda^N B_{n,m}  \arrow{rrr}{(Re^{\max},Cf^{\max})} \& \& \&
	\bigsqcup_{\stackrel{\la \subset n \times m}{|\la|=N}}\mathcal{B}^{\mgl_n}_{\la} \otimes \mathcal{B}^{\mgl_m}_{\la^{\text{tr}}} \\
	\end{tikzcd}
	\vspace{-0.5cm}
	\]
	Starting from the top left, going first right and then down we get:
	\[
	\sii \circ (Re^{\max},Cf^{\max})(P)=\sii (P \otimes Cf^{\max}(P))= \sii(P) \otimes Cf^{\max}(P)
	\]
	Whereas, applying first $\so$ and then the horizontal map gives us:
	\begin{align*}
		&(Re^{\max},Cf^{\max}) \circ \so(P)=Re^{\max}(\so(P)) \otimes Cf^{\max}(\so(P))= \\ &=\so(Re^{\max}(P)) \otimes Cf^{\max}(\so(P))=\so(P) \otimes Cf^{\max}(\so(P))
	\end{align*}
	
	So, we need to verify that:
	
	\begin{enumerate}
		\item $\so(P)=\sii(P)$
		\item $Cf^{\max}(\so(P))=Cf^{\max}(P)$
	\end{enumerate}
	
	For (1), both $\so$ and $\sii$ are determined by and change only the top $k$ rows of $P$, so we will restrict our attention to them and call that submatrix $P_{\left[k\right]}$. Consider $L=Ce^{\max}(P_{\left[k\right]})=Ce_{i_1} \hdots Ce_{i_{s-1}}Ce_{i_s}(P_{\left[k\right]})$ where $\{i_1, \; i_2, \hdots, \; i_s\}$ is some path to the $\mgl_k$--highest weight element associated to $P_{\left[k\right]}$. Then $P_{\left[k\right]}=Cf_{i_s} Cf_{i_{s-1}}\hdots Cf_{i_1}(L)$, and so using Lemma \ref{lem:parcom}:
	\begin{align*}
		\so(P_{\left[k\right]}) &=\so(Cf_{i_s} Cf_{i_{s-1}} \hdots Cf_{i_1}(L))=Ce_{k-i_s}\so(Cf_{i_{s-1}} \hdots Cf_{i_1}(L))= \\
		&= \hdots = Ce_{k-i_s} Ce_{k-i_{s-1}}\hdots Ce_{k-i_1}\so(L)
	\end{align*}
	And similarly, $\sii(P_{\left[k\right]})=Ce_{k-i_s} Ce_{k-i_{s-1}}\hdots Ce_{k-i_1}\sii(L)$ by the definition of the inner action in Theorem \ref{prop:action}. So, it suffices to check that $\so(L)=\sii(L)$. Since $L$ is both $\mgl_m$-- and $\mgl_k$--highest weight, as discussed in Lemma \ref{lem:part} it is a matrix filled with zeros except the upper left corner, where $1$'s fill the shape of a partition $\la$ with $\la^{\text{tr}}=(\sum^k_{i=1}{P_{i,1}}, \sum^k_{i=1}{P_{i,2}}, \hdots, \sum^k_{i=1}{P_{i,m}})$. 
	
	In any crystal of the form $\mathcal{B}^{\mgl_m}_{\varpi_{i_1}} \otimes \hdots \otimes \mathcal{B}^{\mgl_m}_{\varpi_{i_k}}$, let $b^{\text{high}}_j$ denote the highest weight element in $\mathcal{B}^{\mgl_m}_{\varpi_{i_j}}$, i.e. $\wt(b^{\text{high}}_j)= \varpi_{i_j}$, and $b^{\text{low}}_j$ denote the lowest weight element, $\wt(b^{\text{low}}_j)= w^{\mgl_m}_0 \cdot\varpi_{i_j}$. We can consider $L \in \Lambda^{\ast} {B_{k,m}} \cong \bigsqcup_{\underline{i}} \mathcal{B}^{\mgl_m}_{\varpi_{i_1}} \otimes \hdots \otimes \mathcal{B}^{\mgl_m}_{\varpi_{i_k}}$, by looking at the rows of $L$. Under this correspondence, we can express $L$ as $b_1 \otimes b_2 \otimes \hdots \otimes b_k$, where $b_j \in (\{0,1\})^m$ is the $j$--th row of $L$. Furthermore, $L$ having entries $1$ filling the shape of the partition $\la=(\la_1,\hdots,\la_k)$ implies that $b_1 \otimes b_2 \otimes \hdots \otimes b_k  \in \mathcal{B}^{\mgl_m}_{\varpi_{\la_1}} \otimes \hdots \otimes \mathcal{B}^{\mgl_m}_{\varpi_{\la_k}}$ and each $b_i$ is a $\mgl_m$--highest weight element in $\mathcal{B}^{\mgl_m}_{\varpi_{\la_i}}$, i.e. $b_i=b^{\text{high}}_i$ and $\wt(b_i)=\varpi_{\la_i}$. Therefore, $\xi(b_i)=b^{\text{low}}_i \in \mathcal{B}^{\mgl_m}_{\varpi_{\la_i}}$, where $b^{\text{low}}_i$ with $\wt(b^{\text{low}}_i)=w^{\mgl_m}_0 \cdot \varpi_{\la_i}$ is the lowest weight element in that crystal. Then:
	\begin{align*}
	\so(L) &=\so(b^{\text{high}}_1 \otimes b^{\text{high}}_2 \otimes \hdots \otimes b^{\text{high}}_k) \\
	&= \xi(\xi(b^{\text{high}}_k) \otimes \hdots \otimes \xi(b^{\text{high}}_2) \otimes \xi(b^{\text{high}}_1)) 
	=\xi(b^{\text{low}}_k \otimes \hdots \otimes b^{\text{low}}_2 \otimes b^{\text{low}}_1)
	\end{align*}
	For the final computation, we will consider the weights of the elements in the crystal we are working with.  For any other element $b=b_1\otimes b_2 \otimes \hdots \otimes b_k$ in  $\mathcal{B}^{\mgl_m}_{\varpi_{i_1}} \otimes \hdots \otimes \mathcal{B}^{\mgl_m}_{\varpi_{i_k}}$:
	\[
	\sum^k_{j=1}{w^{\mgl_m}_0 \cdot \varpi_{i_j}}\leq\wt(b)=\sum^k_{j=1}{\wt(b_j)} \leq \sum^k_{j=1}{\varpi_{i_j}}
	\]
	We get equality on the left precisely when $b_j=b^{\text{low}}_j \; \forall \; j=1,\hdots,k$ and equality on the right precisely when $b_j=b^{\text{high}}_j \; \forall \; j=1,\hdots,k$. Namely, $b^{\text{low}}_1 \otimes b^{\text{low}}_2 \otimes \hdots \otimes b^{\text{low}}_k$ and $b^{\text{high}}_1 \otimes b^{\text{high}}_2 \otimes \hdots \otimes b^{\text{high}}_k$ are the unique elements in $\mathcal{B}^{\mgl_m}_{\varpi_{i_1}} \otimes \hdots \otimes \mathcal{B}^{\mgl_m}_{\varpi_{i_k}}$ with weights $\sum^k_{j=1}{w^{\mgl_m}_0 \cdot \varpi_{i_j}}$ and $\sum^k_{j=1}{\varpi_{i_j}}$ respectively.
	
	Back to our case, since:
	\begin{align*}\wt(\xi(b^{\text{low}}_k \otimes \hdots \otimes b^{\text{low}}_2 \otimes b^{\text{low}}_1)) &=w^{\mgl_m}_0 \cdot\wt(b^{\text{low}}_k \otimes \hdots \otimes b^{\text{low}}_2 \otimes b^{\text{low}}_1) \\
		&=
	w^{\mgl_m}_0 \cdot \sum^k_{i=1}{w^{\mgl_m}_0 \cdot \varpi_{\la_i}} = \sum^k_{i=1}{\varpi_{\la_i}}
	\end{align*}
	we must have $\xi(b^{\text{low}}_k \otimes \hdots \otimes b^{\text{low}}_2 \otimes b^{\text{low}}_1)=b^{\text{high}}_k \otimes \hdots \otimes b^{\text{high}}_2 \otimes b^{\text{high}}_1$.
	
	So, we get that for a matrix $L$ with filling determined by the partition $\la=(\la_1,\la_2,\hdots,\la_k)$ as above:
	\begin{center}$\so(L)=$ the $k \times m$ matrix with entries equal to zero except in the lower left corner which has $1$'s filling the shape of the reflected partition  $(\la_k,\hdots,\la_2,\la_1)$ (see Figure \ref{fig:sll})\end{center}
	
	\begin{figure}[h]
		$$\begin{tikzpicture}
		\matrix[matrix of math nodes,left delimiter={[}, right delimiter={]}](m)
		{	
			\phantom{0} & \phantom{0} & \phantom{0} &  \la_1  & \phantom{0} & \phantom{0} & \phantom{0} & \phantom{0} \\
			\phantom{0} & \phantom{0} & \phantom{0} &  \la_2 & \phantom{0} & \phantom{0} & \phantom{0} & \phantom{0} \\
			\phantom{0} & \phantom{0} & \phantom{0} & \vdots & \phantom{0} & \phantom{0} & \phantom{0} & \phantom{0}  \\	
			\phantom{0} &  \phantom{0} & \phantom{0} & \phantom{0}  & \phantom{0} & \phantom{0} & \phantom{0} & \phantom{0} \\
			\phantom{0} & \phantom{0} & \phantom{0}  & \la_k & \phantom{0} & \phantom{0} & \phantom{0} & \phantom{0}  \\
		};
		\draw[color=red](m-1-1.north west) -- (m-5-1.north west);
		\draw[color=red](m-1-1.north west) -- (m-1-8.north west);
		\draw[color=red] (m-1-8.north west) -- (m-2-8.north west);
		\draw[color=red] (m-2-8.north west) -- (m-2-7.north west);
		\draw[color=red] (m-2-7.north west) -- (m-3-7.north west);
		\draw[color=red] (m-3-7.north west) -- (m-3-5.north west);
		\draw[color=red] (m-3-5.north west) -- (m-4-5.north west);
		\draw[color=red] (m-4-5.north west) -- (m-4-3.north west);
		\draw[color=red] (m-4-3.north west) -- (m-5-3.north west);
		\draw[color=red] (m-5-3.north west) -- (m-5-1.north west);
		\node[left] at ($(m.west)+(-0.5,0)$) {$L=$};
		\node[right] at ($(m.east)+(0.5,0)$) {$\stackrel{\so}{\xrightarrow{\hspace*{1cm}}}$};
		
		\begin{scope}[xshift=6.5cm]
		\matrix[matrix of math nodes,left delimiter={[}, right delimiter={]}](n)
		{
			\phantom{0} &  \phantom{0} & \phantom{0} & \la_k  & \phantom{0} & \phantom{0} & \phantom{0} & \phantom{0} \\
			\phantom{0} & \phantom{0} & \phantom{0}  & \phantom{0} & \phantom{0} & \phantom{0} & \phantom{0} & \phantom{0}  \\
			\phantom{0} & \phantom{0} &  \phantom{0} & \vdots  &\phantom{0} & \phantom{0} & \phantom{0} & \phantom{0}  \\
			\phantom{0} & \phantom{0} & \phantom{0} & \la_2 & \phantom{0} & \phantom{0} & \phantom{0} & \phantom{0} \\
			\phantom{0} & \phantom{0} &  \phantom{0} &\la_1 & \phantom{0} & \phantom{0} & \phantom{0} & \phantom{0} \\
		};
		\draw[color=red](n-5-1.south west) -- (n-1-1.south west);
		\draw[color=red](n-5-1.south west) -- (n-5-8.south west);
		\draw[color=red] (n-5-8.south west) -- (n-4-8.south west);
		\draw[color=red] (n-4-8.south west) -- (n-4-7.south west);
		\draw[color=red] (n-4-7.south west) -- (n-3-7.south west);
		\draw[color=red] (n-3-7.south west) -- (n-3-5.south west);
		\draw[color=red] (n-3-5.south west) -- (n-2-5.south west);
		\draw[color=red] (n-2-5.south west) -- (n-2-3.south west);
		\draw[color=red] (n-2-3.south west) -- (n-1-3.south west);
		\draw[color=red] (n-1-3.south west) -- (n-1-1.south west);
		\end{scope}
		\end{tikzpicture}$$
		\caption{The result of applying $\so$ to a $\mgl_k$-- and $\mgl_m$--highest weight matrix.}\label{fig:sll}
	\end{figure}

	For the action of $\sii$, we consider $L \in \Lambda^{\ast} {B_{k,m}} \cong \bigsqcup_{\underline{j}}\mathcal{B}^{\mgl_k}_{\varpi_{j_1}} \otimes \hdots \otimes  \mathcal{B}^{\mgl_k}_{\varpi_{j_m}}$ as a tensor product of $\mgl_k$--crystals, using the column decomposition. Let $\la^{\text{tr}}=:\tau =(\tau_1,\tau_2, \hdots,\tau_m)$ be the conjugate partition of the partition $\la$ described above. Then, an analogous argument about the filling of $L$ as before tells us that $L=c^{\text{high}}_m \otimes c^{\text{high}}_{m-1} \otimes \hdots \otimes c^{\text{high}}_1 \in \mathcal{B}^{\mgl_k}_{\varpi_{\tau_m}} \otimes \hdots \otimes  \mathcal{B}^{\mgl_k}_{\varpi_{\tau_1}}$, with $\wt(c^{\text{high}}_i)=\varpi_{\tau_i}$. By weight considerations similar to the earlier ones: 
	\begin{align*}
		\wt(\sii(L)) &= \wt(\sii(c^{\text{high}}_m \otimes c^{\text{high}}_{m-1} \otimes \hdots \otimes c^{\text{high}}_1)) \\
		&=w^{\mgl_k}_0 \cdot \wt(c^{\text{high}}_m \otimes c^{\text{high}}_{m-1} \otimes \hdots \otimes c^{\text{high}}_1) \\
		&= w^{\mgl_k}_0 \cdot \sum^m_{i=1}\tau_i = \sum^m_{i=1}w^{\mgl_k}_0 \cdot \tau_i = \wt(c^{\text{low}}_m \otimes c^{\text{low}}_{m-1} \otimes \hdots \otimes c^{\text{low}}_1)
	\end{align*}
	And as before, since $c^{\text{low}}_m \otimes c^{\text{low}}_{m-1} \otimes \hdots \otimes c^{\text{low}}_1$ is the unique element in the crystal $\mathcal{B}^{\mgl_k}_{\varpi_{\tau_m}} \otimes \hdots \otimes  \mathcal{B}^{\mgl_k}_{\varpi_{\tau_1}}$ with weight $\sum_i\tau_i$, $\sii(L)=c^{\text{low}}_m \otimes c^{\text{low}}_{m-1} \otimes \hdots \otimes c^{\text{low}}_1$.
	Since $c^{\text{low}}=(c^{\text{high}} \text{ reflected through the middle})$, i.e. $(c^{\text{low}})_i= (c^{\text{high}})_{k+1-i}$ for all $i=1,\hdots,k$, we get:
	\begin{center}$\sii(L)=$ the matrix $L$ reflected through the middle row, i.e. the $k \times m$ matrix with entries equal to zero except in the lower left corner which has $1$'s filling the shape of the reflected partition  $(\la_k,\hdots,\la_2,\la_1)$\end{center}
	This operation gives the same result on $L$ as applying $\so$ (shown in Figure \ref{fig:sll}), thus $\so(L)=\sii(L)$ and so $\so(P)=\sii(P)$.
	
	Next we show (2) $Cf^{\max}(\so(P))=Cf^{\max}(P)$. we know that $\so$ is a $\mgl_m$--crystal isomorphism (as shown in \cite{HK}), so takes $\mgl_m$--highest weight elements to $\mgl_m$--highest weight elements, namely $\so(P)$ is also a $\mgl_m$--highest weight element. So, $Cf^{\max}(P)$ and $Cf^{\max}(\so(P))$ are both obtained by moving all the way to the bottom all the $1$'s in $\so(P)$, respectively $P$, and obtaining matrices with $0$'s everywhere except in the lower left corner where the $1$'s fill an upside-down partition, i.e. aligned Southwest. All we have to confirm then, is that the action of $\so$ preserves the number of $1$'s in each column of $P$, which then determine the partition. But since $\so$ is a $\mgl_m$--crystal isomorphism, it preserves the $\mgl_m$--crystal weights $(\sum^n_{i=1}P_{i,1}, \hdots, \sum^n_{i=1}P_{i,m})$ and $\sum^n_{i=1}P_{i,j}$ is the number of $1$'s in the $j$--th column of $P$, which is therefore preserved by $\so$.
\end{proof}

\begin{example} We show that the inner and outer actions of $s_{1,2}$ agree on the matrix from Example \ref{ex:tabl}.
	\begin{align*}s^{\text{o}}_{1,2}&:M=\left[ \begin{array}{ccccc}
			1 & 1 & 1 & 0 & 0  \\
			0 & 0 & 1 & 1 & 0  \\
			1 & 1 & 1 & 0 & 1  \\
		\end{array}\right] \xrightarrow{\text{flip}^{1,2}} \left[ \begin{array}{ccccc}
		0 & 0 & 1 & 1 & 0  \\
		1 & 1 & 1 & 0 & 0  \\
		1 & 1 & 1 & 0 & 1  \\
	\end{array}\right] \\
	&\xrightarrow{\xi \otimes \xi \otimes \text{Id}} \left[ \begin{array}{ccccc}
	0 & 1 & 1 & 0 & 0  \\
	0 & 0 & 1 & 1 & 1  \\
	1 & 1 & 1 & 0 & 1  \\
\end{array}\right] 
\xrightarrow{\xi\otimes\text{Id}} \left[\begin{array}{ccccc}
	1 & 0 & 1 & 0 & 0  \\
	0 & 1 & 1 & 1 & 0  \\
	1 & 1 & 1 & 0 & 1  \\
\end{array}\right]=s^{\text{i}}_{1,2}(M)\end{align*}
\end{example}

\begin{corollary}
	The outer action of $C_m$ on the $m$--tensor product of $\mgl_n$--crystals agrees with the inner action on the $\mgl_m$--crystal $\Lambda^KB_{n,m}$ under the crystal isomorphism: $\Lambda^KB_{n,m} \cong \bigsqcup_{\stackrel{\la \subset n \times m}{|\la|=K}}\mathcal{B}^{\mgl_n}_{\la} \otimes \mathcal{B}^{\mgl_m}_{\la^{\text{tr}}}$.  More precisely, for any $s_{p,q} \in C_m$ and any $N   \in \Lambda^K{B_{n,m}}$:
	\[s^{\text{o}}_{m+1-q,m+1-p}(N)=s^{\text{i}}_{p,q}(N)\]
\end{corollary}

\begin{proof}
	Consider the map $F: \Lambda^K{B_{m,n}} \longrightarrow \Lambda^K{B_{n,m}}$ which rotates a given matrix counterclockwise by $90$ degrees. It is a bijection, and we will consider its effect on the ``outer'' and ``inner'' cactus group actions. We will first show that for $M \in \Lambda^K{B_{m,n}}$, $Re_i(F(M))=F(Ce_i(M)) \; \forall \; i=1,\hdots, m-1$. Let $N=F(M)$ and first consider:
	\begin{align*}
		R\ep^k_i(N)&=\sum^{k-1}_{s=1}{(N_{s,i+1}-N_{s,i})}+\delta_{N_{k,i},0}\delta_{N_{k,i+1},1} \\
		&=\sum^{k-1}_{s=1}{(M_{i+1,n+1-s}-M_{i,n+1-s})}+\delta_{M_{i,n+1-k},0}\delta_{M_{i+1,n+1-k},1} \\
		&=\sum^{n}_{t=n+2-k}{(M_{i+1,t}-M_{i,t})}+\delta_{M_{i,n+1-k},0}\delta_{M_{i+1,n+1-k},1}=C\ep^k_i(M)
	\end{align*}
	So, if $Ce_i$ acts in the $(n+1-k)$--th column in $M$, changing the pair of elements $(M_{i,n+1-k},M_{i+1,n+1-k})=(N_{k,i},N_{k,i+1})$ from $(0,1)$ to $(1,0)$, then $Re_i$ acts in the $k$--th row of $N$, changing $(N_{k,i},N_{k,i+1})$ from $(0,1)$ to $(1,0)$. Therefore, $Re_i(F(M))=F(Ce_i(M))$ and analogously, we can show that $Rf_i(F(M))=F(Cf_i(M))$.
	We will study the inner cactus group action first. Given $M \in \Lambda^K{B_{m,n}}$ and $s_{p,q} \in C_m$, there exist two decompositions:
	\[Cf_{i_s}\hdots Cf_{i_1}(H) = M = Ce_{j_t} \hdots Ce_{j_1}(L)\]
	for some $i_1,\hdots, i_s,j_1,\hdots, j_t \in \{p,\hdots, q-1\}$ and $H \in \Lambda^K{B_{m,n}}$ such that $Ce_i(H)=0 \; \forall \; i \in \{p,\hdots,q-1\}$, $L \in \Lambda^K{B_{m,n}}$ such that $Cf_i(L)=0 \; \forall \; i \in \{p,\hdots,q-1\}$. Then let $d=q+1-p$, and using that $s^{\text{i}}_{p,q}(Cf_i(M))=Ce_{d-i}(s^{\text{i}}_{p,q}(M)), \; s^{\text{i}}_{p,q}(Ce_i(M))=Cf_{d-i}(s^{\text{i}}_{p,q}(M)), \; s^{\text{i}}_{p,q}(H)=L$ from Definition \ref{def:schu} of the Sch\"{u}tzenberger involution, we get:
	$$\begin{array}{ccc}
	M &\xrightarrow{s^{\text{i}}_{p,q}} Ce_{d-i_s} \hdots Ce_{d-i_1} L &\xrightarrow{F} Re_{d-i_s} \hdots Re_{d-i_1} F(L) \\
	M &\xrightarrow{F} Rf_{i_s} \hdots Rf_{i_1} F(H) &\xrightarrow{s^{\text{i}}_{p,q}} Re_{d-i_s} \hdots Re_{d-i_1} F(L)
	\end{array}$$
	The last step above follows since: $$Re_i(F(H))=F(Ce_i(H))=0=F(Cf_i(L))=Rf_i(F(L)) \; \forall \; i \in \{p,\hdots,q-1\}$$ so $F(H)$ and $F(L)$ are the highest  and lowest weight elements of an irreducible component of the crystal restricted to the interval $\{p,q\}$, and so they are switched by $s^{\text{i}}_{p,q}$. Overall:
	\begin{equation}
	\label{eqn:si}
	F(s^{\text{i}}_{p,q}(M))=s^{\text{i}}_{p,q}(F(M))
	\end{equation}
	Next, we consider the outer cactus group action. Since the ``outer'' action of an element $s_{p,q} \in C_m$ on $M \in \Lambda^K{B_{m,n}}$ is determined by the $\mgl_n$--crystal structure, we will study that. It in turn is determined by the rows of $M$, denoted $r_1,\hdots, r_m$. Under the decomposition in Equation (\ref{eqn:row}) on page 9, $M$ corresponds to $r_1 \otimes \hdots \otimes r_m$ as a $\mgl_n$--crystal element. For $N=F(M)$, we have $N_{k,j}=M_{j,n+1-k}$. Hence, the $j$--th column of $N$ corresponds to the reflected $j$--th row of $M$, which by Lemma \ref{lem:flip} is $\xi(r_j)$. So, the columns of $N$ are $\xi(r_1), \hdots,\xi(r_m)$. Under the decomposition in Equation (\ref{eqn:col}), $N$ then corresponds to $\xi(r_m) \otimes \hdots \otimes \xi(r_1)$ as a $\mgl_n$--crystal element. Diagrammatically,
	\[\left[\begin{array}{ccc}
	\textrm{\bf{---}} & r_1 & \longrightarrow \\
	& \vdots & \\
	\textrm{\bf{---}} & r_m & \longrightarrow
	\end{array}\right]
	\xrightarrow{\phantom{0} F \phantom{0}} \left[\begin{array}{ccc}
	\uparrow & & \uparrow \\
	r_1 & \hdots & r_m \\
	| & & | 
	\end{array}\right]=
	\left[\begin{array}{ccc}
	| & & | \\
	\xi(r_1) & \hdots & \xi(r_m) \\
	\downarrow & & \downarrow 
	\end{array}\right]\]
	So the map $F$ on the level of $\mgl_n$--crystals can be described as $\xi^{\otimes m} \circ \text{flip}^{1,m}$: \[F(r_1\otimes \hdots \otimes r_m)=\xi(r_m) \otimes \hdots \otimes \xi(r_1)\]
	Composing $F$ with the outer action of the cactus group, we have:

	\begin{align*}
	& r_1\otimes \hdots \otimes r_m \xrightarrow{s^{\text{o}}_{p,q}} 
	r_1 \otimes \hdots \otimes r_{p-1} \otimes \xi(\xi(r_q) \otimes \hdots \otimes \xi(r_p)) \otimes r_{q+1} \otimes \hdots \otimes r_m \\
	&(\widetilde{r_q} \otimes \hdots \otimes \widetilde{r_p}:= \xi(\xi(r_q) \otimes \hdots \otimes \xi(r_p))) \\
	&\xrightarrow{F} 
	\xi(r_m) \otimes \hdots \otimes \xi(r_{q+1}) \otimes \xi(\widetilde{r_p}) \otimes \hdots \otimes \xi(\widetilde{r_q}) \otimes\xi(r_{p-1}) \otimes \hdots \otimes \xi(r_1)
	\end{align*}

	\begin{align*}
	&r_1\otimes \hdots \otimes r_m \xrightarrow{F} \xi(r_m) \otimes \hdots \otimes \xi(r_q) \otimes \hdots \otimes \xi(r_p) \otimes \hdots \otimes \xi(r_1) \\
	& (\widetilde{q}=m+1-q, \; \widetilde{p}=m+1-p) \\
	&\xrightarrow{s^{\text{o}}_{\widetilde{q},\widetilde{p}}} \xi(r_m) \otimes \hdots \otimes \xi(r_{q+1}) \otimes \xi(r_p \otimes \hdots \otimes r_q) \otimes \xi(r_{p-1}) \otimes \hdots \otimes \xi(r_1)
	\end{align*}
	
	\vspace{0.3cm}	
	The first $(m-q)$ elements (i.e. $\xi(r_m) \otimes \hdots \otimes \xi(r_{q+1})$) and the last $(p-1)$ elements (i.e. $\xi(r_{p-1}) \otimes \hdots \otimes \xi(r_1)$) of the resulting tensor product agree under both compositions, so we look at what happens to $r_p \otimes \hdots \otimes r_q \in B_p \otimes \hdots \otimes B_q$ (denoting the $\mgl_n$--crystals to which these elements belong). They end up in positions $m+1-q$ to $m+1-p$ in the overall tensor product in both compositions, and there the action of the first composition $F \circ s^{\text{o}}_{p,q}$ is:
	\[(\xi_{B_p} \otimes \hdots \otimes \xi_{B_q})\circ\text{flip}^{p,q}\circ\xi_{B_q \otimes \hdots\otimes B_p}\circ(\xi_{B_q} \otimes \hdots \otimes \xi_{B_p})\circ\text{flip}^{p,q}\]
	Similarly, we can describe the second composition $s^{\text{o}}_{m+1-q,m+1-p} \circ F$ acting on $r_p \otimes \hdots \otimes r_q$ as:
	\begin{align*}\xi_{B_p \otimes \hdots \otimes B_q}\circ(\xi_{B_p} \otimes \hdots \otimes \xi_{B_q})\circ\text{flip}^{p,q}\circ (\xi_{B_q} \otimes \hdots \otimes \xi_{B_p}) \circ \text{flip}^{p,q} \\
		=\xi_{B_p \otimes \hdots \otimes B_q}\circ\text{flip}^{p,q}\circ(\xi_{B_q} \otimes \hdots \otimes \xi_{B_p})\circ (\xi_{B_q} \otimes \hdots \otimes \xi_{B_p}) \circ \text{flip}^{p,q} \\
		=(\xi_{B_p} \otimes \hdots \otimes \xi_{B_q})\circ\text{flip}^{p,q}\circ\xi_{B_q \otimes \hdots\otimes B_p}\circ (\xi_{B_q} \otimes \hdots \otimes \xi_{B_p}) \circ \text{flip}^{p,q}
	\end{align*}
	The last equality follows from Proposition 2 in \cite{HK}, where it is shown that $\xi_{B_p \otimes \hdots \otimes B_q}\circ\text{flip}^{p,q}\circ(\xi_{B_q} \otimes \hdots \otimes \xi_{B_p})=(\xi_{B_p} \otimes \hdots \otimes \xi_{B_q})\circ\text{flip}^{p,q}\circ\xi_{B_q \otimes \hdots\otimes B_p}$ as part of the result that the outer cactus group action is an involution. This tells us:
	\begin{equation}
	\label{eqn:so}
	F(s^{\text{o}}_{p,q}(M))=s^{\text{o}}_{m+1-q,m+1-p}(F(M))
	\end{equation}
	Finally, given any matrix $N \in \Lambda^K{B_{n,m}}$, there is a matrix $M \in \Lambda^K{B_{m,n}}$ such that $F(M)=N$. From Theorem \ref{thm:agree}, for an element $s_{p,q} \in C_m$ we know that $s^{\text{o}}_{p,q}(M)=s^{\text{i}}_{p,q}(M)$. Applying $F$ to both sides and using Equations (\ref{eqn:si}) and (\ref{eqn:so}), we get:
	\[s^{\text{o}}_{m+1-q,m+1-p}(N)=s^{\text{i}}_{p,q}(N).\]
\end{proof}
\begin{example} Using the $3 \times 5$ matrix $M \in \Lambda^9B_{3,5}$ from Example \ref{ex:tabl} and $s_{4,5} \in C_5$, we show that the actions of $s^{\text{o}}_{1,2}$ and $s^{\text{in}}_{4,5}$ agree:
	
	\begin{align*}s^{\text{o}}_{1,2}&:M=\left[ \begin{array}{ccccc}
			1 & 1 & 1 & 0 & 0  \\
			0 & 0 & 1 & 1 & 0  \\
			1 & 1 & 1 & 0 & 1  \\
		\end{array}\right] \xrightarrow{\text{flip}^{1,2}} \left[ \begin{array}{ccccc}
		1 & 1 & 1 & 0 & 0  \\
		0 & 0 & 1 & 0 & 1  \\
		1 & 1 & 1 & 1 & 0  \\
	\end{array}\right] \\
	& \xrightarrow{\xi \otimes \xi \otimes \text{Id}^{\otimes 3}} \left[ \begin{array}{ccccc}
	1 & 1 & 1 & 1 & 0  \\
	0 & 0 & 1 & 0 & 1  \\
	1 & 1 & 1 & 0 & 0  \\
\end{array}\right]	\xrightarrow{\xi\otimes\text{Id}^{\otimes 3}} \left[ \begin{array}{ccccc}
		1 & 1 & 1 & 0 & 0  \\
		0 & 0 & 1 & 1 & 0  \\
		1 & 1 & 1 & 0 & 1  \\
	\end{array}\right]=s^{\text{i}}_{4,5}(M)\end{align*}	

\end{example}

In the next section we discuss one setting in which crystal structures and the action of the cactus group can be applied.

\section{Shift of Argument and Gaudin Algebras}
\label{sec:shiftgaudin}
A further application of the two cactus group actions is relating two families of commutative subalgebras of $U(\fg)$ and $U(\fg)^{\otimes n}$. We describe this below using skew Howe duality for $\mgl_n$ and $\mgl_m$.

\subsection{Shift of argument algebras}
The shift of argument algebras are a family $\{\mathcal{A}_{\mu}\}$ of maximal commutative subalgebras of $U(\fg)$, parametrized by elements $\mu$ in the de Concini--Procesi wonderful compactification $\overline{\mathbb{P}\left(\fh^{\text{reg}}_{\R}\right)}$ where $\fh^{\text{reg}}_{\R}$ are the regular, real points of $\fh$ (\cite{S}, \cite{AFV}, \cite{HKRW}). A result in \cite{DJS} relates this space to the pure cactus group, for any Lie algebra $\fg$:
\[PC_{\fg} \cong \pi_1\left(\overline{\mathbb{P}\left(\fh^{\text{reg}}_{\R}\right)}\right)\]

For any irreducible highest weight $\fg$-representation $V_{\la}$ with corresponding crystal $\mathcal{B}_{\la}$, work by Rybnikov leads to the result that $\mathcal{A}_{\mu} \subset U(\fg)$ acts on $V_{\la}$ with simple spectrum (\cite{FFR10}, \cite{HKRW}). This leads to a covering $E(\la) \xrightarrow{\phi_{\la}} \overline{\mathbb{P}\left(\fh^{\text{reg}}_{\R}\right)}$ and a monodromy action of the fundamental group $\pi_1(\overline{\mathbb{P}\left(\fh^{\text{reg}}_{\R}\right)})=PC_{\fg}$ on a fiber $\phi^{-1}_{\la}(\mu)$ of $\mathcal{A}_{\mu}$-eigenlines. Joint work with J. Kamnitzer, L. Rybnikov and A. Weekes relates this geometric monodromy action of the pure cactus group with the combinatorial cactus group action on $\mathcal{B}_{\la}$ defined in Section \ref{sec:cactus} :

\begin{theorem}[\cite{HKRW}]
	\label{conj:shift}
	For $\fg$ a finite-dimensional complex reductive Lie algebra, $\la \in \Lambda_+$, and $\mu \in \overline{\mathbb{P}\left(\fh^{\text{reg}}_{\R}\right)}$: \begin{center}$\mathcal{B}_{\la}\cong \phi^{-1}_{\la}(\mu)$ as $PC_{\fg}$--sets\end{center}
	\label{thm:shift}
\end{theorem}

Note that in the case of $\fg=\mgl(n,\C)$, $\overline{\mathbb{P}\left(\fh^{\text{reg}}_{\R}\right)}=\overline{M^{n+1}_0(\R)}$ where $\overline{M^{n+1}_0(\R)}$ is the Deligne--Knudson--Mumford moduli space of stable real curves of genus $0$ with $n+1$ marked points, also appearing in the work of Kapranov \cite{K} and Devadoss \cite{D}. So, we have $PC_n  \cong \pi_1\left(\overline{\mathbb{P}\left(\fh^{\text{reg}}_{\R}\right)} \right) \cong \pi_1\left(\overline{M^{n+1}_0(\R)}\right)$.

\subsection{Gaudin algebras} Given an $n$--tuple $\underline{z}=(z_1,z_2,\hdots,z_n) \in \C^n$ of pairwise distinct complex numbers, the Gaudin algebra $G_{\underline{z}}$ is generated by the Gaudin Hamiltonians (using the Casimir) and more complicated expressions (using the Feigin--Frenkel \cite{FF} center at the critical level) \cite{FFR94}. In \cite{R14}, it is shown that $G_{\underline{z}}$ acts with simple spectrum on $\text{Hom}(V_{\nu},V_{\la_1} \otimes \hdots \otimes V_{\la_n})$ for any $(n+1)$--tuple of highest weights $(\nu,\underline{\la})=(\nu,(\la_1,\hdots,\la_n))$. Let $B_{\nu}, B_{\la_1}, \hdots, B_{\la_n}$ denote the crystals for those representations and consider the image of $\underline{z}$ in $M^{n+1}_0(\mathbb{R})$. Analogously to the shift of argument setting, this leads to a covering $E(\nu,\underline{\la}) \xrightarrow{\psi_{\nu,\underline{\la}}} \overline{M^{n+1}_0(\mathbb{R})}$ and so a monodromy action $\pi_1(\overline{M^{n+1}_0(\mathbb{R})})=PC_n(=PC_{\mgl_n})$ on a fiber $\psi^{-1}_{\nu,\underline{\la}}(\underline{z})$ of $\mathcal{G}_{\underline{z}}$--eigenlines, which coincides with the outer cactus group action of $C_n$ on an $n$-tensor product of crystals defined in \cite{HK}:

\begin{theorem}(\cite{HKRW})
	\label{conj:gaudin}
	For $\fg$ a reductive Lie algebra, $\nu, \underline{\la} \subset \Lambda_+$, and $\underline{z} \in \overline{M^{n+1}_0(\mathbb{R})}$: \begin{center}$\text{Hom}_{\fg-\text{crystals}}(B_{\nu}, B_{\la_1} \otimes \hdots \otimes B_{\la_n})\setminus \{0\}\cong \psi^{-1}_{\nu,\underline{\la}}(\underline{z})$ as $PC_n$--sets\end{center}
\end{theorem}

\subsection{An application of skew Howe duality} Classical skew Howe duality on $\Lambda^N (\C^n \otimes \C^m)$ allows us to consider this space as both a $\mgl_n$-- and $\mgl_m$--representation, and we have the decomposition:
\[\Lambda^N (\C^n \otimes \C^m) \cong \bigoplus_{\underline{k} \in \mathbb{N}^n, |\underline{k}|=N} \Lambda^{k_1} \mathbb{C}^m \otimes \hdots \otimes \Lambda^{k_n}\mathbb{C}^m\]
So, given $\underline{z} \in \overline{M^{n+1}_0(\mathbb{R})}$, we can consider the action of both $\mathcal{A}_{\underline{z}} \subset U(\mgl_n)$ and $G_{\underline{z}} \subset U(\mgl_m)^{\otimes n}$ on $W=\Lambda^{k_1} \mathbb{C}^m \otimes \hdots \otimes \Lambda^{k_n}\mathbb{C}^m$. Indeed, since each $\Lambda^k \C^m$ is a $\mgl_m$--representation, $G_{\underline{z}}$ acts on $W$, and $\mathcal{A}_{\underline{z}}$ commutes with the action of $U(\mathfrak{h})$ and so preserves the $\mgl_n$--weight space $W$ with weight $(k_1,\hdots,k_n)$. The images of the two actions $a: \mathcal{A}_{\underline{z}} \longrightarrow \text{End}(W)$ and $g: G_{\underline{z}} \longrightarrow \text{End}(W)$ are expected to agree  in $\text{End}(W)$. So, we expect to get the same fiber of eigenlines over any $\underline{z} \in \overline{M^{n+1}_0(\mathbb{R})}$ for $\mathcal{A}_{\underline{z}}$ and $G_{\underline{z}}$, and therefore the same cover. Considering the fibers as $PC_n$--sets, we use Theorem \ref{thm:agree} which tells us that the outer cactus action of $PC_n$ on tensor products of crystals agrees with the inner action (as $PC_n$ also preserves the weight spaces in a $\mgl_n$--crystal) under skew Howe duality for $W$, so the fibers are also the same as $PC_n$--sets. This relates Theorems \ref{conj:shift} and \ref{conj:gaudin} in this $(\mgl_n,\mgl_m)$ setting.

	\end{document}

%% file: sl3crystalnew.pdf_tex
\begingroup%
  \makeatletter%
  \providecommand\color[2][]{%
    \errmessage{(Inkscape) Color is used for the text in Inkscape, but the package 'color.sty' is not loaded}%
    \renewcommand\color[2][]{}%
  }%
  \providecommand\transparent[1]{%
    \errmessage{(Inkscape) Transparency is used (non-zero) for the text in Inkscape, but the package 'transparent.sty' is not loaded}%
    \renewcommand\transparent[1]{}%
  }%
  \providecommand\rotatebox[2]{#2}%
  \ifx\svgwidth\undefined%
    \setlength{\unitlength}{78.90312735bp}%
    \ifx\svgscale\undefined%
      \relax%
    \else%
      \setlength{\unitlength}{\unitlength * \real{\svgscale}}%
    \fi%
  \else%
    \setlength{\unitlength}{\svgwidth}%
  \fi%
  \global\let\svgwidth\undefined%
  \global\let\svgscale\undefined%
  \makeatother%
  \begin{picture}(1,1.14326307)%
    \put(0,0){\includegraphics[width=\unitlength,page=1]{sl3crystalnew.pdf}}%
    \put(0.33374155,1.1024745){\color[rgb]{0,0,0}\makebox(0,0)[lb]{\smash{${\scriptstyle \alpha_1+\alpha_2}$}}}%
    \put(0,0.78839756){\color[rgb]{0,0,0}\makebox(0,0)[lb]{\smash{}}}%
    \put(0.07887348,0.74472953){\color[rgb]{0,0,0}\makebox(0,0)[lb]{\smash{\small $\alpha_2$}}}%
    \put(0.80341484,0.74472721){\color[rgb]{0,0,0}\makebox(0,0)[lb]{\smash{\small $\alpha_1$}}}%
    \put(0.0034627,0.31309132){\color[rgb]{0,0,0}\makebox(0,0)[lb]{\smash{\small $-\alpha_1$}}}%
    \put(0.80387974,0.31309441){\color[rgb]{0,0,0}\makebox(0,0)[lb]{\smash{\small $-\alpha_2$}}}%
    \put(0.64148656,0.51766578){\color[rgb]{0,0,0}\makebox(0,0)[lb]{\smash{\small $0$}}}%
    \put(0.27459147,0.51093702){\color[rgb]{0,0,0}\makebox(0,0)[lb]{\smash{$\scriptstyle 0$}}}%
    \put(0.28864586,-0.02169622){\color[rgb]{0,0,0}\makebox(0,0)[lb]{\smash{\small $-\alpha_1-\alpha_2$}}}%
    \put(0,0){\includegraphics[width=\unitlength,page=2]{sl3crystalnew.pdf}}%
    \put(0.2828879,0.89315496){\color[rgb]{0,0,0}\makebox(0,0)[lb]{\smash{\small $1$}}}%
    \put(0.71813733,0.61829074){\color[rgb]{0,0,0}\makebox(0,0)[lb]{\smash{\small $1$}}}%
    \put(0.27404844,0.41810615){\color[rgb]{0,0,0}\makebox(0,0)[lb]{\smash{\small $1$}}}%
    \put(0.66744226,0.14940198){\color[rgb]{0,0,0}\makebox(0,0)[lb]{\smash{\small $1$}}}%
    \put(0.64438163,0.89183777){\color[rgb]{0,0,0}\makebox(0,0)[lb]{\smash{\small $2$ }}}%
    \put(0.19901979,0.61829074){\color[rgb]{0,0,0}\makebox(0,0)[lb]{\smash{\small $2$ }}}%
    \put(0.62688619,0.41810166){\color[rgb]{0,0,0}\makebox(0,0)[lb]{\smash{\small $2$ }}}%
    \put(0.25174264,0.14940492){\color[rgb]{0,0,0}\makebox(0,0)[lb]{\smash{\small $2$ }}}%
    \put(0.35486557,0.48422709){\color[rgb]{0,0,0}\makebox(0,0)[lt]{\begin{minipage}{0.05069508\unitlength}\raggedright \end{minipage}}}%
  \end{picture}%
\endgroup%